\documentclass{amsart}

\usepackage[utf8]{inputenc}
\usepackage{graphicx}
\usepackage{graphicx,color}
\usepackage{latexsym}
\usepackage[all]{xy}
\usepackage{mathrsfs}
\usepackage{enumerate}
\usepackage[shortlabels]{enumitem}
\usepackage{multicol}
\usepackage{lipsum}
\usepackage{amssymb}
\usepackage{tikz}
\usetikzlibrary{cd}
\usetikzlibrary{patterns}
\definecolor{DimGray}{rgb}{0.41, 0.41, 0.41}
\usepackage{float}
\usepackage{bm}
\usepackage{mathptmx}
\usetikzlibrary{shapes.geometric}
\usepackage{dsfont}
\usepackage{cjhebrew}
\usepackage[alpine]{ifsym}
\usepackage{xpicture}
\usepackage{calculator}
\usepackage{graphicx}

\usepackage{pgf,tikz,pgfplots}
\pgfplotsset{compat=1.15}
\usetikzlibrary{arrows}
\usetikzlibrary{patterns}

\usepackage{amsmath}
\usepackage{amsfonts}
\usepackage{amssymb,enumerate}
\usepackage{amsthm}
\usepackage{fancyhdr}
\usepackage{hyperref}
\usepackage{cleveref}

\usepackage{caption}
\usepackage{subcaption}

\usepackage[backend=biber,style=numeric, maxbibnames=99,doi=false,isbn=false,url=false]{biblatex}
\renewbibmacro{in:}{%
	\ifentrytype{article}{}{\printtext{\bibstring{in}\intitlepunct}}}

\newtheorem{theorem}{Theorem}[section]

\newtheorem{proposition}[theorem]{Proposition}
\newtheorem*{theorem*}{Theorem}

\newtheorem{corollary}[theorem]{Corollary}
\newtheorem{conj}[theorem]{Conjecture}

\theoremstyle{definition}
\newtheorem{defin}[theorem]{Definition}
\newtheorem{ex}[theorem]{Example}

\newtheorem{question}[theorem]{Question}

\newtheorem{prob}[theorem]{Problem}

\theoremstyle{remark}
\newtheorem{rem}[theorem]{Remark}

\numberwithin{equation}{section}

\DeclareMathOperator{\Ap}{Ap}

\def\NN{{\mathbb N}}

\def\lcm{\operatorname{lcm}}

\newcommand{\obeta}{a}

\makeatletter
\newcommand{\mysubsection}[1]{%
	\medskip\noindent\textbf{#1}
	\@afterindentfalse\@afterheading
}
\makeatother

\begin{document}

\title[Complete intersection monomial curves]{ On the Rich Landscape of Complete Intersection Monomial Curves}


\author{Patricio Almir\'on}
\address{Departamento de \'Algebra, An\'alisis Matem\'atico, Geometr\'ia y Topolog\'ia; IMUVA (Instituto de Investigaci\'on en Matem\'aticas), Universidad de Valladolid\\ 
Paseo de Bel\'en 7\\
47011 Valladolid, Spain.}

\email{palmiron@uva.es}
\thanks{The author was supported by Grant RYC2021-034300-I funded by MICIU/AEI/10.13039/501100011033 and by European Union NextGenerationEU/PRTR and during the elaboration of this work was also supported by Spanish Ministerio de Ciencia, Innovaci\'{o}n y Universidades PID2020-114750GB-C32 and by the IMAG–Maria de Maeztu grant CEX2020-001105-M / AEI /10.13039/501100011033, through a postdoctoral contract in the ‘Maria de Maeztu Programme for Centres of Excellence’. }

\date{}

\dedicatory{Dedicated with gratitude to Alejandro Melle Hernández on the occasion of his 55 birthday.}


\begin{abstract}
	The aim of this survey is to explore complete intersection monomial curves from a contemporary perspective. The main goal is to help readers understand the intricate connections within the field and its potential applications. The properties of any monomial curve singularity will be first reviewed, highlighting the interaction between combinatorial and algebraic properties. Next, we will revisit the two main characterizations of complete intersection monomial curves. One is based on deep algebraic properties given by Herzog and Kunz in 1971, while the other is based on a combinatorial approach given by Delorme in 1976. Our aim is to bridge the gap between these perspectives present in the current literature. Then, we will focus on recent advances that show an intriguing connection between numerical semigroups and Alexander polynomials of knots. Finally, we will revisit the deformation theory of these curves.
\end{abstract}

\maketitle
\tableofcontents

\section{Introduction}
Let \((C,0)\subset(\mathbb{C}^g,0)\) be a germ of irreducible complex curve singularity. Let \(R\) denotes its local ring at the origin and let \(\overline{R}\simeq\mathbb{C}[\, t]\) be its normalization. The normalization \(R\hookrightarrow\overline{R}\) induces a discrete valuation \(v:R\rightarrow\mathbb{Z}\) from which \(S=v(R)\) has a natural structure of finitely generated subsemigroup of the natural numbers, \(\mathbb{N},\) with \(0\) element. The semigroup \(S\) is usually referred as the value semigroup of the curve. One can consider the associated graded algebra \mbox{\(\operatorname{gr}(R)=\oplus_{n\geq 0}I_n/I_{n+1}\)} of \(R,\) where \(I_n=M^n\cap R\) and \(M\) is the maximal ideal of \(\overline{R}.\) Then, it is easy to check that this associated graded ring is equal to the semigroup algebra of \(S,\) i.e.
\[\operatorname{gr}(R)=\bigoplus_{s\in S}\mathbb{C}t^s=\mathbb{C}[S].\]
Studying  \(\operatorname{gr}(R)\) instead of \(R\) offers advantages. Since the combinatorial structure of \(S\) aligns more closely with the algebraic properties of \(\operatorname{gr}(R)\) than \(R\) itself, purely combinatorial properties of the semigroup \(S\) can be translated into algebraic properties of \(\operatorname{gr}(R)\), and vice versa. However, there's a price to pay: in most cases some information specific to the ring structure of \(R\) might be lost. This is not the case for monomial curve singularities, as for them \(\operatorname{gr}(R)\) and \(R\) contains the same algebraic information.
\medskip

Let us recall that, following Herzog's terminology \cite{Herzogthesis}, a geometric semigroup \(S\) is a finitely generated semigroup satisfying \(S\cap (-S) =\{0\}\). An algebraic variety \(V\subseteq\mathbb{A}^n_\mathbb{C}\) is called a monomial variety if its coordinate ring is isomorphic to the semigroup ring over \(\mathbb{C}\) of a geometric semigroup \(S.\) Therefore, for a monomial curve singularity we will refer to an irreducible monomial variety of complex dimension \(1\) whose semigroup algebra is isomorphic to a geometric semigroup \(S\subset \mathbb{N}\) over the natural numbers. Consequently, the study of the algebraic and geometric properties of the curve can be approached through understanding the combinatorial structure of its semigroup of values. 
\medskip

From our point of view, monomial curves can be thought of as the geometric representation of a numerical semigroup. Among them, we believe that those that are complete intersections most faithfully represent the richness of the interactions between the combinatorics of the semigroup and the geometry of the curve. Thus, this class of semigroups provides a rich context to motivate the study of these interactions in more general cases. The structure of this survey is designed as an introduction to the study of monomial curves from a broad and transversal perspective.  Obviously, this perspective is biased by the author's personal vision, so we will mainly focus on some aspects that we believe may be of interest to delve deeper into with the new techniques available today. Likewise, we suggest new problems in the field that we believe broaden the already transversal research in this area.
\medskip

Inspired by Buchweitz's 1977 survey about monomial curves \cite{Buchweitz}, Section \ref{sec:monomial} gives an overview of general properties of monomial curves, updating Buchweitz's work with some advances from the last four decades. Two of these new contributions deserve special emphasis. First, a valuable contribution is the description of the minimal free resolution of the semigroup algebra associated with a numerical semigroup provided by Campillo and Marijuan in 1991 \cite{CM}. We believe that their description based on the simplicial complexes defined by Székely and Wormald in 1986 \cite{SWshaded} has considerable potential for further exploration. Second, over the last 30 years, Rosales, García-Sánchez and their many collaborators, among others, have significantly contributed to the development of purely combinatorial properties of numerical semigroups and related structures (see \cite{AAGbook,RGbook} and references therein). This development focuses on basic integer arithmetic, avoiding tools from commutative algebra and algebraic geometry. As a result, many recently discovered properties remain uninterpreted in a geometric or algebraic context. We believe it is crucial to contextualize them within their origins to effectively address new challenges in the field.
\medskip

The 1970s undoubtedly mark the starting point of the extensive exploration into the close connection between the combinatorial properties of the semigroup of values and the algebraic properties of its associated ring. Initiating this connection were Kunz \cite{KunzGorsym}, who showed the characterization of one-dimensional Gorenstein rings in terms of their value semigroups, and his student, Herzog, whose doctoral thesis \cite{Herzogthesis} made the first attempts to characterize complete intersection one-dimensional rings in terms of their semigroup of values. The development of these results went hand in hand with the evolution of commutative algebra produced by the German school during that decade. Part of this evolution was a result of the application of local duality in the context of commutative algebra as showed in the book edited by Herzog and Kunz \cite{HKbook71}. Following this line, the 1971 paper by Herzog and Kunz \cite{HK71} constitutes a pioneering study of the main relationships between the value semigroup and its associated graded ring. In this work, they provide some characterizations of complete intersection one dimensional rings by using the powerful properties of the Dedekind, Noether and K\"{a}hler differents previously established by Berger \cite{Berger60,Berger63} and Kunz \cite{Kunz68}.
\medskip

From a different angle, during the mid-20th century the development of a dictionary between algebraic and topological properties of singularities has Zariski as one of its greatest figures. His work "Studies in Equisingularity" \cite{ZarstudiesI,ZarstudiesII,ZarstudiesIII} and the development of the concept of saturation of local rings \cite{ZarsaturationI,ZarsaturationII,ZarsaturationIII} have inspired and been the source of countless results in the years since. Zariski's school and its influence is essentially extended to most of the topics in singularity theory. In this context can be framed the 1974 work of Pinkham \cite{Pinkham1} about deformations of an isolated singularity of an algebraic variety admitting a multiplicative group action. Observe that the canonical example of such varieties is precisely a monomial curve singularity, from which \cite[Chapter IV]{Pinkham1} is devoted. A special point of interest of Pinkham's work is \cite[Theorem 13.9]{Pinkham1} the connection between the moduli space of pointed smooth projective algebraic curves with fixed genus and the weights of the cohomology groups of the cotangent complex of \(R\) defined by Lichtenbaum and Schlessinger in 1967 \cite{LS67}. This connection has been extensively explored in the years since, as one can see for instance in \cite{Buchweitz,stohr1,sthor2,RV77} among others. Additionally, Zariski presented an alternative moduli problem during a series of lectures in the fall of 1973 at the \'{E}cole Polytechnique in Paris. This alternative moduli problem focused on irreducible plane curve singularities. The notes of those lectures were collected in the book \cite{Zarbook} with an appendix by Teissier \cite{teissierappen} initiating an outbreak of a deep study of the moduli space of irreducible plane curve singularities. Similarly to the connections established by Pinkham, the results of Teissier \cite{teissierappen} further highlights the utility of deforming monomial curves in order to obtain results about a certain moduli space. 
\medskip

In this setting, Delorme's characterization of complete intersection monomial curve singularities appears in his 1976 paper \cite{delormeglue}, drawing on his doctoral thesis work.  Delorme's results \cite{delormeglue} provide a characterization quite different from the one previously given by Herzog and Kunz in \cite{HK71}. His main methods rely heavily on the combinatorics of the semigroup of values, rather than deeper properties of the associated graded ring. This purely combinatorial characterization has made Delorme's results more widely known, partly because they can be generalized to finitely generated semigroups of \(\mathbb{N}^n\) (see \cite{RosGluing,FMAffine}). Unfortunately, there are no references that explain both characterizations in a unified way, and they are often discussed in a disconnected manner. As a result, we believe that a significant part of the rich structure of complete intersection semigroups is often lost. One of the main motivations of this survey, and which we hope is also one of its main contributions, is precisely to show the parallel between Delorme's results and those of Herzog-Kunz in a unified exposition. We hope that this will serve to better understand the richness of these semigroups.
\medskip

Section~\ref{sec:CI} compares the characterization of complete intersections monomial curves given by Delorme \cite{delormeglue} with that given by Herzog and Kunz \cite{HK71}. As we have already mentioned, although Delorme's characterization is based on basic properties, we will see why it is so useful. We will review some of the results and give some examples that show the great interaction between the degrees associated with a regular sequence and certain decomposition properties of the semigroup algebra. We will also review the characterization of Herzog and Kunz. While their results are applicable to any Noetherian local ring of dimension \(1\), we will primarily focus on the case of the semigroup algebra associated with a monomial curve. In any case, we aim to provide an overview of Herzog and Kunz's results in a way that this specific case could serve as a valuable illustration of their core arguments in the general situation.
\medskip

Having reviewed these results, we can now explore some of the exciting connections that have emerged recently. On one hand, for a given link in the \(3\)--sphere we have the classical polynomial invariant, called the Alexander polynomial, defined by Alexander in 1928 \cite{alexoriginal}. On the other hand, for a given graded algebra, e.g. the semigroup algebra associated to a monomial curve, a classical invariant is the Hilbert-Poincaré series, which is the generating function for the dimensions of its graded pieces. In the early 21st century, Campillo, Delgado and Gusein-Zade \cite{CDG99a, CDG99b, CDG00, CDG02, CDG03a, CDG04, CDG05, CDG07} discovered a fascinating connection between the generating series of certain semigroups appearing in geometric contexts and the Alexander polynomial of some knots and links. In short, they show that there exists several cases where these two completely different invariants coincide up to a certain factor. Section \ref{sec:Poincare} will focus precisely on the connection between Hilbert-Poincaré series and the Alexander polynomial of a knot, with a particular emphasis on the case of a numerical semigroup. 
\medskip

Inspired by this connection, S. Wang introduced the formal semigroup of a knot through its Alexander polynomial in 2018 (see Section~\ref{subsec:poincaresemigroup} for further details). While the formal semigroup is not always a semigroup, Wang proved that it becomes one in specific cases. Section~\ref{sec:Poincare} will revisit these connections and offer new insights.  Furthermore, our recent work with Moyano-Fernández \cite{AMknots} provides a new proof of the Campillo-Delgado-Gusein-Zade theorem \cite{CDGduke} and reveals a deeper connection between the semigroup algebra of the semigroup of an irreducible plane curve singularity and the fundamental group of the knot complement. We will also revisit this connection in Section~\ref{subsec:poincaresemigroup} and propose a more general framework to explore connections for broader classes of numerical semigroups, such as complete intersections.
\medskip

To finish, we elaborate over some facts about deformations of monomial curves and moduli spaces. First, as already mentioned, Pinkham's work \cite{Pinkham1} shows that an important topic related to the study of the base space of the miniversal deformation of a monomial curve is its connection with the moduli space of projective curves with a given Weierstrass semigroup. In section \ref{sec:deformations}, we will recall Pinkham's construction \cite{Pinkham1} as well as Teissier's results \cite{teissierappen}. Secondly, we discovered that, surprisingly, Delorme's results had not been applied to the study of deformations of complete intersection curves. This led to our recent paper \cite{AMminiversal} with Moyano-Fernández where we proved the usefulness of Delorme's theorem for the calculation of an explicit basis of monomials of the base space of the miniversal deformation of a complete intersection monomial curve. On the other hand, Buchweitz and Greuel in \cite{BG80} defined the Milnor number for curves with any embedding dimension and showed that it has topological properties under deformations similar to those of the case of a plane curve. In section \ref{sec:deformations} we will briefly revisit some of those results.
\medskip

In this way, we hope that with this journey through the world of complete intersection monomial curves, exploring both their inherent combinatorial structure and their connections to other topics such as knot theory, the reader will be able to uncover the rich landscape of this fascinating family.


\vspace{0.5 cm}
\noindent \textbf{Acknowledgements}\footnote{During the preparation of this survey one of the key figures behind some of the main results discussed, J\"{u}rgen Herzog, sadly passed away. I hope this work can in some way also honor his memory and contributions.}. The author is extremely grateful to Julio José Moyano-Fernández for the countless hours of conversation and collaboration on this topic. He also pointed me to the articles of the German school and helped me with their translation. All of this has been the main motivation for the preparation of this article. I would also like to thank Jan Stevens for pointing out an error in an earlier version of this survey, in an example related to a question posed by Greuel and Buchweitz.

\section{Overview on affine monomial curves}\label{sec:monomial}

Let \(S\) be a finitely generated subsemigroup of \(\NN.\) Thanks to the isomorphism of semigroups given by \(S\simeq S/\gcd(S),\) we can assume without loss of generality that \(\gcd(S)=1\) and hence \(S\) is a so called numerical semigroup. Observe that the condition \(\gcd(S)=1\) immediately implies \(|\mathbb{N}\setminus S|<\infty\) and the existence of a natural number \(c\in\mathbb{N}\), called the conductor of \(S,\) such that \(c+\mathbb{N}\subset S\) and \(c-1\notin S.\) It is a well know theorem by R\'edei \cite[Fundamental Theorem]{Redei} that any finitely generated commutative semigroup is finitely presented. Following Herzog \cite{herzog}, we will recall how the presentation of a numerical semigroup translate into the defining equations of its monomial curve. 
\medskip

Let us consider \(\{a_1,\dots,a_g\}\) a system of generators (not necessarily minimal) of \(S,\) i.e.
\[S=\{n\in\mathbb{N}\;| \; n=\alpha_1a_1+\cdots+\alpha_g a_g \;\text{with} \; \alpha_i\in\mathbb{N}\}.\]
If \(\mathbf{e}_1,\dots,\mathbf{e}_g\) denotes the canonical basis of \(\NN^g,\) one can naturally define the epimorphism 
\[\begin{array}{cccc}
	\rho:&\NN^g&\rightarrow& S\\
	& \mathbf{e}_i&\mapsto& a_i .
\end{array}
\]
The kernel of this epimorphism provides a binary relation on \(\NN^g\) given by
\[\kappa:=\{(v,w)\in\NN^g\times\NN^g\;|\;\rho(v)=\rho(w) \},\]
which provides the presentation of \(S.\) A direct application of R\'edei's theorem \cite{Redei} shows that one has the isomorphism \(S\simeq \NN^g/\kappa\). It is frequent in the literature to call \(\kappa\) a congruence on \(\NN.\) To be more specific, we will say that \(\kappa\) is a congruence from \(S.\) An element \((v,w)\in\kappa\) will be called syzygy. Our choice of the name syzygy is not innocent at all. To justify this name, we will need to show the relation between a congruence from \(S\) and the presentation of \(\mathbb{C}[S]\) as \(\mathbb{C}[x_1,\dots,x_g]\)--module.
\medskip

Let us consider the ring of polynomials \(\mathbb{C}[x_1,\dots,x_g]\) in \(g\) variables. For an element \mbox{\(v=(v_1,\dots,v_g)\in\NN^g,\)} we define \(\mathbf{x}^v=\prod_{i=1}^{g}x_i^{v_i}.\) Also, we can identify \(\mathbf{x}^v\) with its representation in the semigroup ring \(\mathbf{x}_v\in\mathbb{C}[\NN^g].\) Via this identification, we have a natural isomorphism \(\mathbb{C}[x_1,\dots,x_g]\simeq\mathbb{C}[\NN^g]\) of graded \(\mathbb{C}\)--algebras, with their natural grading. In this way, the epimorphism \(\rho\) of semigroups induces a \(\mathbb{C}\)--algebra epimorphism defined by 
\[\begin{array}{cccc}
	\varphi:&\mathbb{C}[x_1,\dots,x_g]&\rightarrow& \mathbb{C}[S]\\
	& \mathbf{x}^v&\mapsto& \mathbf{x}_{\rho(v)} .
\end{array}
\]

In addition, we can endow the polynomial ring \(\mathbb{C}[x_1,\dots,x_g]\) with an \(S\)--graduation via the map \(\rho\) that sets \(\deg(x_i)=a_i.\) If this is the case, \(\varphi\) becomes an homogeneous homomorphism of degree \(0\) between \(S\)--graded \(\mathbb{C}\)--algebras. Let us denote by \(I_S=\ker(\varphi)\) and for a syzygy \((v,w)\in\kappa\) let us consider their binomial representation in the form 
\[F_{(v,w)}=\mathbf{x}^v-\mathbf{x}^w.\]
The following Theorem due to Herzog \cite{herzog,Herzogthesis} shows that the defining ideal \(I_S\) can be computed from the congruence, as one could expect.
\begin{theorem}\cite[Proposition 1.4]{herzog}\label{thm:Herzogpresentation}
	Under the previous considerations, we have	\(I_S=(\{F_{(v,w)}\}_{(v,w)\in\kappa}).\) In other words, \(\mathbb{C}[S]\simeq\mathbb{C}[x_1,\dots,x_g]/(\{F_{v,w}\}_{(v,w)\in\kappa})\) as \(S\)--graded \(\mathbb{C}\)--algebras.
\end{theorem}

\begin{rem}
	Observe now that any element of a congruence from \(S\) provides a well defined syzygy of \(\mathbb{C}[S]\) as \(S\)--graded  \(\mathbb{C}[x_1,\dots,x_g]\)--module.
\end{rem}

Finally, as we are considering irreducible curves, we can also provide a parameterization as follows. Let \(t\in\mathbb{C}\) be a local coordinate of the germ \((\mathbb{C},0)\) and let \((x_1,\dots,x_g)\in \mathbb{C}^{g}\) be local coordinates of the germ \((\mathbb{C}^{g}, \boldsymbol{0})\). The monomial curve \( (C^S, \boldsymbol{0}) \subset (\mathbb{C}^{g}, \boldsymbol{0}) \) can be defined via the parameterization
\[ 
C^S : x_i = t^{a_i}, \qquad i=1,\ldots, g.
\]
Write $\mathbb{C}[C^{S}]:=\mathbb{C}[t^{\nu}\;:\;\nu\in S],$ which coincides with the semigroup algebra \(\mathbb{C}[S]\) associated to $S$. We will use either notation depending on whether we want to emphasize the geometric or algebraic interpretation.

\subsection{Minimal resolution}\label{subsec:Resolutions}
Given a numerical semigroup \(S,\) a set \(\{a_1<\cdots< a_{g}\}\) is a minimal set of generators if \(a_{i+1}\notin a_1\mathbb{N}+\cdots+\mathbb{N}a_i.\) Let us fix a minimal set of generators \(\{a_1,\dots,a_g\}\) of a numerical semigroup and consider the semigroup algebra \(\mathbb{C}[S].\) As the semigroup \(S\) is fixed, we will generally consider the ring of polynomials \(\mathbb{C}[x_1,\dots,x_g]\) with the \(S\)--grading \(\deg(x_i)=a_i\). In order to avoid confusion on the different graded \(\mathbb{C}\)--algebra structures, we will denote by \(M(S):=\mathbb{C}[x_1,\dots,x_{g}]\) this \(S\)--graded \(\mathbb{C}\)--algebra and we will sometimes refer to it as the monomial algebra associated to \(S\). 
\medskip

The semigroup algebra \(\mathbb{C}[S]\) is also a graded \(M(S)\)--module, so it makes sense to consider its minimal graded free resolution as \(M(S)\)--module. It is then natural to try to understand this resolution from the combinatorics of the semigroup \(S.\) We refer to \cite{SMbook,BrunsHerzogbook,Peevabook} for generalities about resolutions of graded algebras.
\medskip

By the Auslander-Buchsbaum Theorem (see for example \cite[Theorem 1.3.3]{BrunsHerzogbook}), the graded free resolution of \(\mathbb{C}[S]\) as a graded \(M(S)\)--module is a finite graded resolution of length \(g-1,\)
\begin{equation*}
	\mathcal{F}_{\bullet}:\quad	0\rightarrow M(S)^{l_{g-1}}\xrightarrow{\varphi_{g-1}}\rightarrow\cdots \rightarrow M(S)^{l_1}\xrightarrow{\varphi_1}M(S)\xrightarrow{\varphi_0=\varphi} \mathbb{C}[S]\rightarrow 0,
\end{equation*}
where the map \(\varphi\) is a degree zero homomorphism defined by \(x_i\mapsto t^{a_i}.\) Due to the exactness of the resolution and the graded Nakayama lemma, the submodules \(\ker\varphi_{i}\) are minimally generated by \(l_{i+1}\) homogeneous elements. Let us denote by \(N_i:=\ker\varphi_{i}=(f_{i,1},\dots,f_{i,l_{i+1}})\) and, for each \(f_{i,j},\) let us denote by \(m_{i,j}=\deg(f_{i,j})\) its degree as element of \(M(S).\) In order to make \(\varphi_i\) a degree zero homomorphism one should just endow \(M(S)^{l_i}\) with the following grading:
\[(M(S)^{l_i})_{m}=M(S)_{m-m_{i,1}}\times\cdots \times M(S)_{m-m_{i,l_i}}.\]
In this way, if \(\mathfrak{m}_{M(S)}\) denotes the irrelevant maximal ideal, the graded Betti numbers associated to the free resolution \(\mathcal{F}_\bullet\) are those defined by 
\[\beta_{i,m}=\dim_\mathbb{C}\frac{(N_i)_m}{(\mathfrak{m}_{M(S)}N_i)_m}.\]

Campillo and Marijuan \cite{CM} showed that \(\beta_{i,m}\) can be computed as homological degrees of an adequate simplicial complex. Before to define these simplicial complexes, let us first consider the Koszul complex for the regular sequence \(x_1,\dots,x_g\) in \(M(S):\)
\[\overline{\mathcal{K}}_\bullet:\quad	0\rightarrow \bigwedge^{g} M(S)^{g}\rightarrow \cdots \rightarrow M(S)^{g}\rightarrow M(S)\rightarrow \mathbb{C}[S]\rightarrow 0.\]
It is however more convenient to work with the complex \(\mathcal{K}_\bullet =\overline{\mathcal{K}}_\bullet\otimes \mathbb{C}[S]\). Obviously, the complex \(\mathcal{K}_\bullet\) is nothing but the Koszul complex for the sequence \(t^{a_1},\dots,t^{a_g}\) in \(\mathbb{C}[S].\) With the obvious graduation, the Koszul complex \(\mathcal{K}_\bullet\) is a graded one, so it gives rise to a family of complexes of vector spaces 
\[\mathcal{K}_\bullet,m:\quad	0\rightarrow \mathcal{K}_{g,m}\rightarrow \mathcal{K}_{g-1,m}\rightarrow\cdots \rightarrow \mathcal{K}_{1,m}\rightarrow \mathbb{C}[S]_{m}\rightarrow 0.\] 
Thus, the graded Betti numbers can be computed from the homological degrees of \(\mathcal{K}_\bullet,\) 
\[\beta_{i,m}=\dim_\mathbb{C}H_i(\mathcal{K}_{\bullet,m}).\]

As showed in \cite{CM}, in order to compute \(\beta_{i,m}\) it is enough to understand the graded structure of the vector spaces \(\mathcal{K}_{p,m}.\) To do so, Campillo and Marijuan use the following simplicial complex, which is usually called shaded set. For \(m\in S,\)
\[\Delta_m:=\{I\subset \{1,\dots,g\}\ |\ m-\sum_{i\in I}a_i\in S\}.\]
It is obvious from the definition that \(\Delta_m\) is an abstract simplicial complex in the vertex set \(\{1,\dots,g\}.\) The shaded set \(\Delta_m\) was first defined by Sz\'{e}kely and Wormald in \cite{SWshaded} to try to understand the combinatorics of the generating series of the semigroup \(S\) (see also Section \ref{subsec:poincaresemigroup}).

Campillo and Marijuan showed that one can express the graded structure of \(\mathcal{K}_{p,m}\) in terms of the shaded sets as follows:

\[\mathcal{K}_{p,m}=\bigoplus_{\raisebox{0.5pt}{\scriptsize $\begin{array}{c}
			J\in \Delta_m\\
			\operatorname{Card}(J)=p
\end{array}$}}\mathbb{C}[S]_{m-a_J}\mathbf{e}_J,\]
where \(\mathbf{e}_J=\sum_{i\in J}\mathbf{e}_i\) and \(\mathbf{e}_1,\dots,\mathbf{e}_g\) is the canonical basis of \(\mathbb{C}^g.\) Therefore, if \(\widetilde{H}_{i}(\Delta_m)\) denotes the reduced homology (see for example \cite[Chapter 0]{StanleyBook} or \cite[Chapter 1]{SMbook}), we have the following identifications
\begin{theorem}\cite{CM}
	Let \(\mathbb{C}[S]\) be a semigroup algebra. Under the previous considerations, the graded Betti numbers can be computed as
	\[\beta_{i,m}=\dim H_{i}(\mathcal{K}_{\bullet,m})=\dim \widetilde{H}_{i-1}(\Delta_m).\]
\end{theorem}

Moreover, the topology of those simplicial complexes is relevant for deciding which generators of the ideal $I_S$ are minimal. In fact, the minimal generators of \(I_S\) are precisely those with non trivial \(0\)--homology group, as showed by Briales, Campillo, Marijuan, Pis\'on in  \cite[Theorem 2.5]{BCMP2}; see also \cite[Corollary 9.3]{SMbook} and \cite{CP93}.

\begin{corollary}\cite[Theorem 2.5]{BCMP2}\label{noconexo}
	The ideal $I_S$ has a minimal generator in degree $\mathbf{b}$ if and only if the simplicial complex $\Delta_{\mathbf{b}}$ is disconnected.
\end{corollary}

\begin{ex}\label{ex:simplComplex}
	Let us consider the numerical semigroup \(S=\langle 5,7,9\rangle\). We can compute the minimal free resolution, which is in fact implemented in SINGULAR \cite{singular}:
	\[0\rightarrow M(S)^2\rightarrow M(S)^3\xrightarrow{\varphi_{1}} M(S)\xrightarrow{\varphi_{0}}\mathbb{C}[S]\rightarrow0,\]
	where \(\ker\varphi_0=(f_1=y^2-xz,f_2=z^3-x^4y,f_3=x^5-yz^2)\) and \(\ker\varphi_1=(z^2f_1-xf_2+yf_3,x^4f_1-yf_2+zf_3).\) Observe that the relevant degrees which provide non-zero Betti numbers are \(\beta_{0,14}=\beta_{0,25}=\beta_{0,27}=1\) and \(\beta_{1,32}=\beta_{1,34}=1.\) In Figure \ref{fig:simplicialcomplexes} we show all simplicial complexes for each \(s\in S,\) these simplicial complexes can be also computed with the help of GAP \cite{GAP4}. As expected, the only ones with \(\dim \widetilde{H}_0(\Delta_m)\neq 0\) are those corresponding to the degrees of the generators of \(\ker\varphi_0\) and the only ones with \(\dim \widetilde{H}_1(\Delta_m)\neq 0\) are those corresponding to the generators of \(\ker\varphi_1.\) 
	
	\begin{figure}
		\centering
		
		\begin{tikzpicture}[point/.style={circle,thick,draw=black,fill=black,inner sep=0pt,minimum width=4pt,minimum height=4pt}, scale=0.6]
			\begin{scope}[shift={(0,5)}]
				
				\node (v3) [label=above:{$\emptyset$}] at (0.5,0.5) {};
				
				\node [below] at (0.5,-1) {0};
			\end{scope}
			
			\begin{scope}[shift={(3,5)}]
				\node (v1) [point,label=below:{$5$}] at (0,0) {};
				
				\node [below] at (0.5,-1) {5};
			\end{scope}
			
			\begin{scope}[shift={(6,5)}]
				
				\node (v2) [point,label=below:{$7$}] at (1,0) {};

				\node [below] at (0.5,-1) {7};
			\end{scope}
			
			\begin{scope}[shift={(9,5)}]
				
				\node (v2) [point,label=below:{$9$}] at (1,0) {};

				\node [below] at (0.5,-1) {9};
			\end{scope}
			
			\begin{scope}[shift={(12,5)}]
				
				\node (v3) [point,label=above:{$5$}] at (0.5,1.5) {};
				
				\node [below] at (0.5,-1) {10};
			\end{scope}
			
			\begin{scope}[shift={(15,5)}]
				\node (v1) [point,label=below:{$5$}] at (0,0) {};
				\node (v2) [point,label=below:{$7$}] at (1,0) {};

				\draw (v1) -- (v2);
				
				\node [below] at (0.5,-1) {12};
			\end{scope}

			\begin{scope}[shift={(18,5)}]
				\node (v1) [point,label=below:{$5$}] at (0,0) {};
				\node (v2) [point,label=below:{$7$}] at (1,0) {};
				\node (v3) [point,label=above:{$9$}] at (0.5,1.5) {};
				
				\draw (v1) -- (v3);
				\node [below] at (0.5,-1) {14};
			\end{scope}

			\begin{scope}[shift={(0,0)}]
				\node (v1) [point,label=below:{$5$}] at (0,0) {};
				
				\node [below] at (0.5,-1) {15};
			\end{scope}
			
			\begin{scope}[shift={(3,0)}]
				
				\node (v2) [point,label=below:{$7$}] at (1,0) {};
				\node (v3) [point,label=above:{$9$}] at (0.5,1.5) {};
				
				\draw (v2) -- (v3);
				\node [below] at (0.5,-1) {16};
			\end{scope}
			
			\begin{scope}[shift={(6,0)}]
				\node (v1) [point,label=below:{$5$}] at (0,0) {};
				\node (v2) [point,label=below:{$7$}] at (1,0) {};

				\draw (v1) -- (v2);
				\node [below] at (0.5,-1) {17};
			\end{scope}
			
			\begin{scope}[shift={(9,0)}]
				
				\node (v3) [point,label=above:{$9$}] at (0.5,1.5) {};
				
				\node [below] at (0.5,-1) {18};
			\end{scope}
			
			\begin{scope}[shift={(12,0)}]
				\node (v1) [point,label=below:{$5$}] at (0,0) {};
				\node (v2) [point,label=below:{$7$}] at (1,0) {};
				\node (v3) [point,label=above:{$9$}] at (0.5,1.5) {};
				
				\draw (v1) -- (v2);
				\draw (v1) -- (v3);
				\node [below] at (0.5,-1) {19};
			\end{scope}
			
			\begin{scope}[shift={(15,0)}]
				\node (v1) [point,label=below:{$5$}] at (0,0) {};
				
				\node [below] at (0.5,-1) {20};
			\end{scope}
			
			\begin{scope}[shift={(18,0)}]
				\node (v1) [point,label=below:{$5$}] at (0,0) {};
				\node (v2) [point,label=below:{$7$}] at (1,0) {};
				\node (v3) [point,label=above:{$9$}] at (0.5,1.5) {};
				
				\draw[line width=0.5mm] (v1) -- (v2) -- (v3)--(v1)-- cycle;
				\draw[pattern=north west  lines] (0,0) -- (1,0) -- (0.5,1.5);
				\node [below] at (0.5,-1) {21};
			\end{scope}
			\begin{scope}[shift={(0,-5)}]
				\node (v1) [point,label=below:{$5$}] at (0,0) {};
				\node (v2) [point,label=below:{$7$}] at (1,0) {};

				\draw (v1) -- (v2) ;
				\node [below] at (0.5,-1) {22};
			\end{scope}
			
			\begin{scope}[shift={(3,-5)}]
				\node (v1) [point,label=below:{$5$}] at (0,0) {};
				\node (v2) [point,label=below:{$7$}] at (1,0) {};
				\node (v3) [point,label=above:{$9$}] at (0.5,1.5) {};
				
				\draw (v1) -- (v3);
				\draw (v2) -- (v3);
				\node [below] at (0.5,-1) {23};
			\end{scope}
			
			\begin{scope}[shift={(6,-5)}]
				\node (v1) [point,label=below:{$5$}] at (0,0) {};
				\node (v2) [point,label=below:{$7$}] at (1,0) {};
				\node (v3) [point,label=above:{$9$}] at (0.5,1.5) {};
				
				\draw (v1) -- (v2) ;
				\draw (v1) -- (v3) ;
				\node [below] at (0.5,-1) {24};
			\end{scope}
			
			\begin{scope}[shift={(9,-5)}]
				\node (v1) [point,label=below:{$5$}] at (0,0) {};
				\node (v2) [point,label=below:{$7$}] at (1,0) {};
				\node (v3) [point,label=above:{$9$}] at (0.5,1.5) {};
				
				\draw (v2) -- (v3);
				\node [below] at (0.5,-1) {25};
			\end{scope}
			
			\begin{scope}[shift={(12,-5)}]
				\node (v1) [point,label=below:{$5$}] at (0,0) {};
				\node (v2) [point,label=below:{$7$}] at (1,0) {};
				\node (v3) [point,label=above:{$9$}] at (0.5,1.5) {};
				
				\draw[line width=0.5mm] (v1) -- (v2) -- (v3)--(v1)-- cycle;
				\draw[pattern=north west  lines] (0,0) -- (1,0) -- (0.5,1.5);
				\node [below] at (0.5,-1) {26};
			\end{scope}
			
			\begin{scope}[shift={(15,-5)}]
				\node (v1) [point,label=below:{$5$}] at (0,0) {};
				\node (v2) [point,label=below:{$7$}] at (1,0) {};
				\node (v3) [point,label=above:{$9$}] at (0.5,1.5) {};
				
				\draw (v1) -- (v2);
				\node [below] at (0.5,-1) {27};
			\end{scope}
			
			\begin{scope}[shift={(18,-5)}]
				\node (v1) [point,label=below:{$5$}] at (0,0) {};
				\node (v2) [point,label=below:{$7$}] at (1,0) {};
				\node (v3) [point,label=above:{$9$}] at (0.5,1.5) {};
				
				\draw[line width=0.5mm] (v1) -- (v2) -- (v3)--(v1)-- cycle;
				\draw[pattern=north west  lines] (0,0) -- (1,0) -- (0.5,1.5);
				\node [below] at (0.5,-1) {28};
			\end{scope}
			
			\begin{scope}[shift={(0,-10)}]
				\node (v1) [point,label=below:{$5$}] at (0,0) {};
				\node (v2) [point,label=below:{$7$}] at (1,0) {};
				\node (v3) [point,label=above:{$9$}] at (0.5,1.5) {};

				\draw (v1) -- (v2);
				\draw (v1) -- (v3);
				\node [below] at (0.5,-1) {29};
			\end{scope}
			\begin{scope}[shift={(3,-10)}]
				\node (v1) [point,label=below:{$5$}] at (0,0) {};
				\node (v2) [point,label=below:{$7$}] at (1,0) {};
				\node (v3) [point,label=above:{$9$}] at (0.5,1.5) {};
				
				\draw[line width=0.5mm] (v1) -- (v2) -- (v3)--(v1)-- cycle;
				\draw[pattern=north west  lines] (0,0) -- (1,0) -- (0.5,1.5);
				\node [below] at (0.5,-1) {30};
			\end{scope}
			
			\begin{scope}[shift={(6,-10)}]
				\node (v1) [point,label=below:{$5$}] at (0,0) {};
				\node (v2) [point,label=below:{$7$}] at (1,0) {};
				\node (v3) [point,label=above:{$9$}] at (0.5,1.5) {};

				\draw[line width=0.5mm] (v1) -- (v2) -- (v3)--(v1)-- cycle;
				\draw[pattern=north west  lines] (0,0) -- (1,0) -- (0.5,1.5);
				\node [below] at (0.5,-1) {31};
			\end{scope}
			
			\begin{scope}[shift={(9,-10)}]
				\node (v1) [point,label=below:{$5$}] at (0,0) {};
				\node (v2) [point,label=below:{$7$}] at (1,0) {};
				\node (v3) [point,label=above:{$9$}] at (0.5,1.5) {};
				
				\draw (v2) -- (v3);
				\draw (v1) -- (v2);
				\draw (v1) -- (v3);
				\node [below] at (0.5,-1) {32};
			\end{scope}
			
			\begin{scope}[shift={(12,-10)}]
				\node (v1) [point,label=below:{$5$}] at (0,0) {};
				\node (v2) [point,label=below:{$7$}] at (1,0) {};
				\node (v3) [point,label=above:{$9$}] at (0.5,1.5) {};

				\draw[line width=0.5mm] (v1) -- (v2) -- (v3)--(v1)-- cycle;
				\draw[pattern=north west  lines] (0,0) -- (1,0) -- (0.5,1.5);
				
				\node [below] at (0.5,-1) {33};
			\end{scope}
			
			\begin{scope}[shift={(15,-10)}]
				\node (v1) [point,label=below:{$5$}] at (0,0) {};
				\node (v2) [point,label=below:{$7$}] at (1,0) {};
				\node (v3) [point,label=above:{$9$}] at (0.5,1.5) {};

				\draw (v1) -- (v2) -- (v3)--(v1)-- cycle;
				
				\node [below] at (0.5,-1) {34};
			\end{scope}
			
			\begin{scope}[shift={(18,-10)}]
				\node (v1) [point,label=below:{$5$}] at (0,0) {};
				\node (v2) [point,label=below:{$7$}] at (1,0) {};
				\node (v3) [point,label=above:{$9$}] at (0.5,1.5) {};

				\draw[line width=0.5mm] (v1) -- (v2) -- (v3)--(v1)-- cycle;
				\draw[pattern=north west  lines] (0,0) -- (1,0) -- (0.5,1.5);
				\node [below] at (0.5,-1) {$m\geq 35$};
			\end{scope}
			\draw (-1,3) -- (21,3);
			\draw (-1,-2) -- (21,-2);
			\draw (-1,-7) -- (21,-7);
			\draw (-1,-12) -- (21,-12);
			
		\end{tikzpicture}
		
		\caption{Simplicial complexes for all the elements of the semigroup \(S=5\mathbb{N}+7\mathbb{N}+9\mathbb{N}.\)}\label{fig:simplicialcomplexes}
	\end{figure}
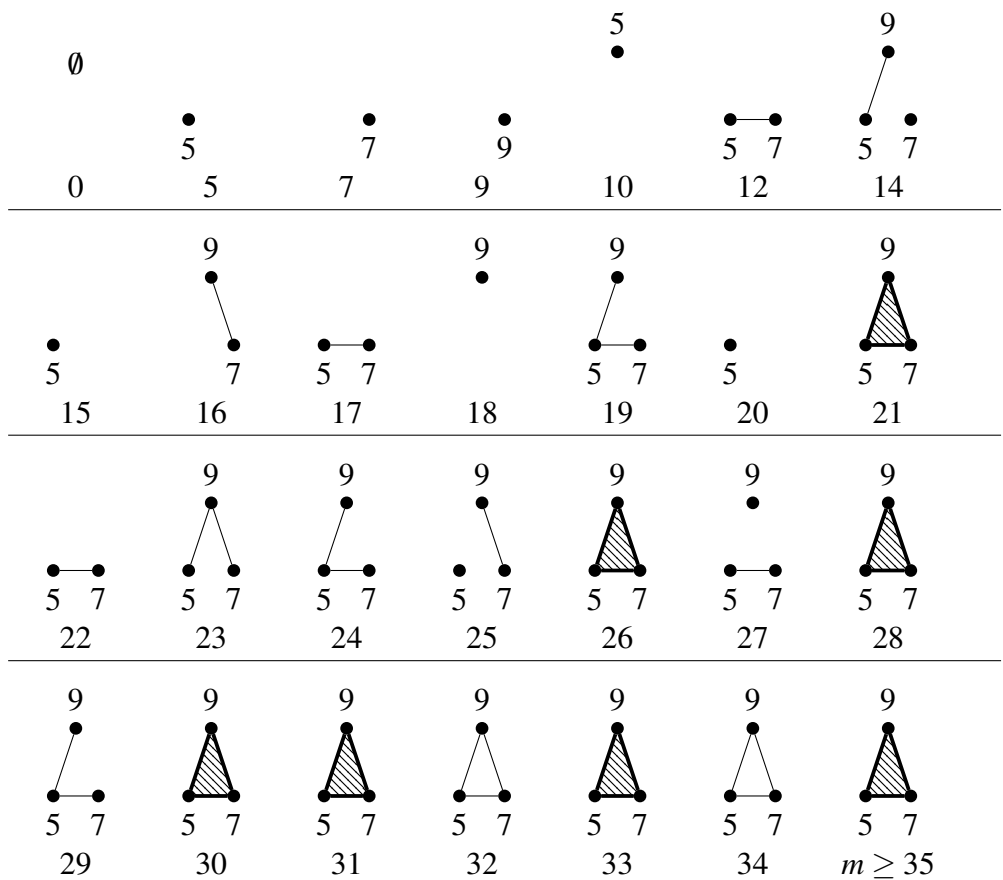
	
\end{ex}

To finish, we must mention that there are some descriptions, due to Bresinsky, of the minimal generators of the defining ideal \(I_S\) in the case of a semigroup with \(3\) and \(4\) generators using different methods to the ones exposed here, see \cite{Bresinsky3,Bresinsky4}. On the other hand, an explicit description in the case of \(5\) generators was given by Campillo and Pisón in \cite{CP93belga} deepening in the possible combinatorial structures that can arise in the simplicial complexes we have already mentioned. Finally, it is also worth noting that Bresinsky \cite{Bresinsky75} also proved that the number generators of \(I_S\) can be arbitrarily large and it is not bounded by the minimal number of generators of \(S.\)

\subsection{Detour on numerical semigroups}\label{sec:numericalsemi}


We have already presented some of the main properties of numerical semigroups, for example that they have a conductor, they are minimally generated, etc. However, it is necessary to introduce some additional concepts that will be useful in the following sections and to also present the main types of numerical semigroups that will appear in the sequel.  For generalities on numerical semigroups the reader is referred to the book of Assi, D'Anna and Garc\'ia-S\'anchez \cite{AAGbook} and the references therein. 
\medskip

Let \(S\) be a numerical semigroup generated by $\{a_1,\ldots , a_g\}$; this fact will be expressed from now on by writing $S=\langle a_1,\ldots , a_g \rangle$. In general, we will not assume that the system to be a minimal system of generators. The elements in the set $\mathbb{N}\setminus S$, which is finite, are called gaps of $S$.  The maximal gap with respect to the usual total ordering in $\mathbb{Z}$ is called the Frobenius number of $S$, written $F(S)$. The number $c(S):=F(S)+1$ is the conductor of $S$, which as we mentioned is the smallest natural number such that \(c+\mathbb{N}\subset S.\) Finally, the delta-invariant of $S$ is defined to be
$$
\delta (S):= |\{x\in S : x < c(S)\}|.
$$
It is important to mention that the conductor and the \(\delta\)--invariant have the following algebraic nature: let \(R=\mathbb{C}[S]\) be the semigroup algebra and let \(\overline{R}=\mathbb{C}[t]\) be its normalization. Then
\[\delta(S)=\dim_\mathbb{C}\frac{\overline{R}}{R},\quad c(S)=\dim_\mathbb{C}\frac{\overline{R}}{(R \; \mathbin{\colon\hspace{-0.15em}\raisebox{-0.4ex}{\scriptsize$R$}} \;  \overline{R})}.\] 
In a similar manner, the dimension of a numerical semigroup \(S\) is defined as the Krull dimension of its semigroup algebra \(\dim S=\dim \mathbb{C}[S].\) In this way, a numerical semigroup always have dimension equal to \(1.\) Also, one can define the rank or embedding dimension of \(S\) as the cardinality of a minimal system of generators of \(S,\) usually denoted as \(e(S).\) It is common to call it embedding dimension as it is the minimal possible embedding dimension of the monomial curve associated to \(S.\) 
\medskip

There is also an alternative but extremely useful system of generators ---by no means minimal--- that can be attached to a numerical semigroup $S$: let $s \in S\setminus \{0\}$, the Ap\'ery set of $S$ with respect to $s$ is defined to be the set
\[
\mathrm{Ap}(S, s)=\{w\in S\,:\;w-s\notin S\}.
\]

\noindent Observe that the cardinality of $\mathrm{Ap}(S, s)$ is $s$, and that $\mathrm{Ap}(S, s)=\{w_0<w_1<\dots<w_{s-1}\}$ where $w_i=\min \{ z\in S : z \equiv i \ \mathrm{mod}\ s\}$; obviously, $w_0=0$ and the Apéry set constitutes naturally a system of generators of \(S.\)
\medskip

Finally, over a numerical semigroup $S$ it is possible to define a module structure in analogy to ring theory. A non-empty subset $\Delta \subseteq \mathbb{Z}$ is said to be a $S$-semimodule or \(S\)--ideal if $\Delta+S \subseteq \Delta$, where the set $\Delta + S$ is understood as all possible sums $a+S$ with $a\in \Delta$, $S\in S$. The set of all $S$-semimodules with respect to this addition has a structure of additive binoid, with $S$ as a neutral element, and $\mathbb{N}$ as an absorbent element; notice that, in particular, $S$ itself is a $S$-semimodule. We will say that an \(S\)--ideal is principal if it is generated by a single element \(h\in \mathbb{Z},\) a principal ideal will be denoted as \((h):=h+S.\) We refer to \cite[\S 1]{HK71} and \cite[Chapter 3]{AAGbook} for generalities about \(S\)--ideals.

\mysubsection{Some classes of numerical semigroups}
From the algebro--geometrical point of view, one of the most important classes of numerical semigroups is the class of symmetric semigroups. A semigroup \(S\) is called symmetric if there exists \(h\in \mathbb{Z}\) such that for all \(x\in S\) we have \(h-x\notin S\) and for all \(y\in\mathbb{Z}\setminus S\) we have \(h-y\in S.\) The following relation between symmetric numerical semigroups and Gorenstein rings was proven by Kunz \cite{KunzGorsym}:
\medskip

\begin{theorem}\cite{KunzGorsym}\label{thm:KunzGorsym}
	Let \(R\) be a one-dimensional analytically irreducible noetherian local ring, \(\overline{R}\) its integral closure in the quotient field \(\mathbb{K}\) and \(v:\mathbb{K}\rightarrow\mathbb{Z}\) the corresponding valuation. Assume \(R\) and \(\overline{R}\) have the same residue class field. Then, \(R\) is Gorenstein if and only if the value semigroup \(v(R)\) is symmetric. 
\end{theorem}

\begin{rem}
	In particular Theorem \ref{thm:KunzGorsym} means that a numerical semigroup is symmetric if and only if its semigroup algebra \(\mathbb{C}[S]\) is Gorenstein. Moreover, it can be shown that the symmetry element is given by \(c-1.\)
\end{rem}

An important class of Gorenstein rings is that of complete intersections, which, roughly speaking, are those that can be defined using the minimum possible number of relations. On the other hand, the definition of a complete intersection semigroup is closely related to its algebraic counterpart. For a given congruence \(\kappa\) from a semigroup \(S,\) the rank of \(\kappa\) is defined as the cardinality of a minimal system of generators of \(\kappa.\) Recall that the previous discussion shows that the rank of \(\kappa\) coincides with the minimal number of generators of \(I_S,\) which will be denoted as \(\operatorname{rank}(I_S).\) Following Herzog \cite{Herzogthesis}, a numerical semigroup \(S\) is called a complete intersection numerical semigoup if 

\[\operatorname{rank}(I_S)=e(S)-1.\]

\begin{rem}
	Observe that from this definition, \(S\) is a complete intersection if and only if \(C^S\) is an ideal theoretic complete intersection variety.
\end{rem}

Theorem \ref{thm:KunzGorsym} shows that for any affine curve (not necessarily a monomial one), the combinatorics of the value semigroup, in this case the symmetry property, determines whether it is Gorenstein or not. Thus, it is now natural to ask whether a similar situation occurs when distinguishing between those that are complete intersections and those that are not. However, in this context the value semigroup is not good enough to provide this classification. The following example \cite{HK71}, shows that not every complete intersection one dimensional local ring (e.g. the coordinate ring of a complete intersection curve singularity) has a complete intersection semigroup of values: 
\begin{ex}\cite[pg. 40/64]{HK71}\label{ex:CIringnotSemi}
	Let \(R=\mathbb{K}[t^6, t^8+2t^9,t^{10}+t^{11}].\) As pointed out by Herzog and Kunz, this is a complete intersection ring. However, one can check that its value semigroup is \(v(R)=\langle 6,8,10,17,19\rangle\) which is not a complete intersection semigroup. We will see in Example \ref{ex:exDelormealgorithm1} why this is not a complete intersection semigroup. 
\end{ex}
Example \ref{ex:CIringnotSemi} is a good example of why the complete intersection property of a ring cannot be directly deduced from its semigroup of values in general. It also reflects the fact that working with the associated graded ring may lead to losing some information. When the curve is not monomial, then, in general, the defining ideal is not generated by homogeneous equations. Thus, the structure of the value semigroup does not perfectly match with the defining equations and hence with the algebraic properties of the ring. If \(R\) is the coordinate ring of a curve singularity (not necessarily monomial),  we can visualize the different relations with the following diagram
\[\begin{matrix}
	R\; \text{is Gorenstein}&\begin{array}{c}
		\Leftarrow\\
		\nRightarrow
	\end{array}&R\;\text{is Complete intersection}\\
	\Updownarrow& &\Uparrow\;\not \Downarrow\\
	\text{Symmetric semigroup}&\begin{array}{c}
		\Leftarrow\\
		\nRightarrow
	\end{array}&\text{Complete intersection semigroup}
\end{matrix}\]

Focusing on complete intersection monomial curves, i.e., complete intersection numerical semigroups, we identify several subclasses that classify them essentially in terms of their defining equations. This has, a priori, no algebraic consequences. However, sometimes it is useful to distinguish them as it facilitates some computations. We refer to \cite{AMSclassesCI} for a detailed discussion about those classes. Here we will just mention the following relevant subclass of complete intersection. Consider a numerical semigroup \(S\) generated (not necessarily minimally) by \(G:=\{ a_1, \dots, a_g\}\). Assume that \(G\) satisfies the condition
\begin{equation}\label{eq:freecond}
	n_ia_i\in\langle a_1, \dots, a_{i-1} \rangle,
\end{equation}
for all \(i=2,\dots,g,\) where \(n_{i}:=\gcd( a_1, \dots, a_{i-1})/\gcd( a_1, \dots, a_i)\). A numerical semigroup admitting a set of generators \(G\) satisfying \eqref{eq:freecond} for all \(i\geq 2\) was named free numerical semigroup by Bertin and Carbonne \cite{beca77}. Moreover, without loss of generality we can further assume that \(n_i>1\) for all \(2\leq i\leq g\). The condition \eqref{eq:freecond}, implies that the equations of \(C^S\) are quite simple and allows easily to check that the corresponding monomial curve is a complete intersection. From the condition \eqref{eq:freecond}, there exist numbers \(\ell_{1}^{(i)},\dots,\ell_{i-1}^{(i)}\in\mathbb{N}\) for each \(i=2,\dots,g\) such that

\[
n_ia_i=\ell_{1}^{(i)}a_1+\cdots+\ell_{i-1}^{(i)}a_{i-1}.
\]
Therefore it is easy to see that the equations of the monomial curve \(C^S\) associated to a free semigroup are of the form
\[
f_i=u_{i}^{n_i}-u_{1}^{\ell_{1}^{(i)}}\cdots u_{i-1}^{\ell_{i-1}^{(i)}}=0\quad\text{for}\;2\leq i\leq g.
\]

\mysubsection{Numerical semigroups associated to irreducible plane curve singularities}
Let us consider \((C,0)\subset(\mathbb{C}^2,0)\) a germ of irreducible plane curve singularity. If \(R\) denotes its local ring, it is well known that the semigroup of values \(S_C=v(R)\) is a numerical semigroup (see \cite{Zarbook}). In fact, in this case the semigroup \(S_C\) is also a complete topological invariant as showed by Zariski, Brauner and Burau \cite{Brauner28,Burau33,Burau34,Zar32}. The minimal generators of \(S_C=\langle a_1,\dots,a_g\rangle\) are computed from a parameterization of the curve (see \cite{Zarbook}) and satisfy the following condition:
\begin{equation}
	\label{cond:beta} S_C\; \text{is free and}\; n_{i}a_i<a_{i+1}\; \text{for}\; 2\leq i\leq g-1.
\end{equation}
In 1972, Bresinsky \cite[Theorem 2]{bresinsky72} and Teissier \cite[Chap. I.~3.2]{teissierappen} independently proved that, for any numerical semigroup \(S\) satisfying condition \eqref{cond:beta} there exists a plane branch \mbox{\((C,\textbf{0})\subset(\mathbb{C}^2,\textbf{0})\)} such that \(S=S_C.\) For this reason, a free semigroup satisfying the conditions \(n_{i}a_i<a_{i+1}\; \text{for}\; 2\leq i\leq g-1\) is usually named as a plane curve semigroup. 
\medskip

On the other hand, one can consider a projective plane curve \(C\subset\mathbb{P}^2_\mathbb{C}.\) If we denote by $L$ the line at infinity, the curve $C$ has only one place at infinity if the intersection $C\cap L$ is a single point $p$ and $C$ has only one analytic branch at $p$. For these curves, one can define the semigroup \( S_{C,\infty}\), usually called the semigroup at infinity:

\[ S_{C,\infty}:=\left\{-\nu_{C,p}(h) \ | \ h\in \mathcal{O}_C(C\setminus\{p\})\right\},\]
where \(\mathcal{O}_C(C\setminus\{p\})\) is the local ring of the curve at the affine chart not containing the place at infinity. It is a very well known theorem by Abhyankar and Moh \cite{AMoh} (see also \cite{Abh3,AbhMoh1,SatSte,suzuki} ) that the semigroup at infinity satisfy the following condition:

\begin{equation}
	\label{cond:delta}  S_{C,\infty}\; \text{is free and}\; n_{i}a_i>a_{i+1}\; \text{for}\; 2\leq i\leq g-1.
\end{equation}
Analogously to the case of irreducible plane curve singularities in \(\mathbb{C}^2,\) Pinkham \cite{Pinkham1,Pinkham} showed that for any semigroup \(S\) satisfying the condition \eqref{cond:delta}, one can find a plane projective curve \(C\) such that \(S=S_{C,\infty}\).

\subsection{Deformations}\label{subsec:deformations}
Given two complex space germs $(X,x)$ and $(S,s)$, a deformation of $(X,x)$ over $(S,s)$ consists of a flat morphism $\phi: (\mathcal{X},x)\to (S,s)$ of complex germs together with an isomorphism from $(X,x)$ to the central fibre of $\phi$, namely an isomorphism $(X,x)\to (\mathcal{X}_s,x):=(\phi^{-1}(s),x).$ Here, $(\mathcal{X},x)$ is called the total space, $(S,s)$ is the base space, and $(\mathcal{X}_s,x)\cong (X,x)$ is the special fibre of the deformation. A deformation is usually denoted by
$$
(i,\phi): (X,x) \overset{i}{\rightarrow} (\mathcal{X},x) \overset{\phi}{\rightarrow} (S,s).
$$
Roughly speaking, a deformation is called versal if any other deformation results from it by a base change. A deformation is called miniversal if it is versal and the base space is of minimal dimension. The existence of a miniversal deformation for isolated singularities is a celebrated result by Grauert \cite{grauert}.
For generalities about deformations we refer to \cite[Chap. II and Appendix C]{Greuelbook} (see also \cite{LS67,Sch,Pinkham1,Buchweitz, HKbook71}).
\medskip

An important object in the study of deformations are the cohomology groups of the cotangent complex of \(\mathbb{C}[S].\) These cohomology groups were defined by Lichtenbaum and Schlessinger in \cite{LS67}. In particular, the first cohomology group \(T^1\) is of special importance as it can be naturally identified with the Zariski tangent space to the base of the miniversal deformation. 
\medskip

Let us now return to the case of monomial curves. In this case, \(T^1\) has a natural \(\mathbb{Z}\)--graded structure. Let \(S=\langle a_1,\dots,a_g \rangle\) be a numerical semigroup and, as in Theorem \ref{thm:Herzogpresentation}, denote by \(I_S\) the defining ideal of the semigroup algebra \(R=\mathbb{C}[S].\) Also, as in section \ref{subsec:Resolutions}, denote by \(M(S)\) the ring of polynomials in \(g\) variables with the grading induced by \(S.\) We have the exact sequence
\[I_S/I^2_S\rightarrow\Omega^1_{M(S)}\otimes_{\raisebox{0.5pt}{\scriptsize $M(S)$}}R\rightarrow\Omega^1_{R}\rightarrow 0\]
which as usual defines the module of K\"{a}hler differentials of \(R.\) This exact sequence yields the following exact sequence:
\[ \operatorname{Hom}_{R}(\Omega^1_{M(S)}\otimes_{\raisebox{0.5pt}{\scriptsize $M(S)$}}R,R)\xrightarrow{\beta}\operatorname{Hom}_{R}(I_S/I_S^2,R)\rightarrow T^1_{R}\rightarrow 0.\]
Observe that \cite[Proposition 1.25]{Greuelbook} we have the following properties:
\begin{enumerate}
	\item \(\operatorname{Hom}_{R}(I_S/I_S^2,R)\simeq \operatorname{Hom}_{M(S)}(I_S,R).\)
	\item With the notations of section \ref{subsec:Resolutions}, \(\beta\) can be identified with the Jacobian matrix 
	\[\left(\frac{\partial f_{0,i}}{\partial x_j}\right)_{1\leq i,j\leq l_1}\] of the defining ideal \(I_S=(f_{0,1},\dots,f_{0,l_1}).\) This is because \(\operatorname{Hom}_{R}(\Omega^1_{M(S)}\otimes_{\raisebox{0.5pt}{\scriptsize $M(S)$}}R,R)\) is a free \(R\)--module generated by the partial derivatives \(\partial/\partial x_i\) and in addition the map \(\beta(\partial/\partial x_i)\in \operatorname{Hom}_{M(S)}(I_S,R) \) is defined as 
	\[I_S\ni h\mapsto \overline{\frac{\partial h}{\partial x_i}}\in R.\]
	\item \(T^1_R\simeq\operatorname{Coker}(\beta).\)
\end{enumerate}

Moreover, if we use the graded free resolution of \(R\) explained in section \ref{subsec:Resolutions}, we have the identification \mbox{\(I_S=\ker\varphi_0=\operatorname{Im}\varphi_1\).} Thus, 
\[\operatorname{Hom}_{M(S)}(I_S,R)=\operatorname{Hom}_{M(S)}(\operatorname{Im}\varphi_1,R).\] With the notations of section \ref{subsec:Resolutions}, let \(\{f_{1,1},\dots,f_{1,l_2}\}\subset M(S)^{l_1}\) be the set of generators of \(N_1,\) i.e. the set of generators for the relations between the elements of \(I_S.\) As we mentioned in section \ref{subsec:Resolutions}, the generators of \(N_1\) can be assumed to be homogeneous with the induced graduation. Following Buchweitz \cite[2.2]{Buchweitz}, the images \(\overline{f}_{1,j}\) of \(f_{1,j}\) modulo \(I_S\) generate \(N_1/I^2_S\) in \(I_S/I_S^2\) as \(R\)--module. Therefore, an element \(\psi\in \operatorname{Hom}_{M(S)}(I_S,R) \) can be identified with a column vector \(h=(h_1,\dots,h_{l_1})^t\) with \(h_i\in R\) such that \(\overline{f}_{1,j}\cdot h=0.\) This identification induces the following graduation in \(\operatorname{Hom}_{M(S)}(I_S,R)\): we say \(\psi\in \operatorname{Hom}_{M(S)}(I_S,R)\) has weight \(\gamma\) if and only if its representative \(h=(h_1,\dots,h_{l_1})^t\) is such that \(h_j\) has weight \(m_{0,j}+\gamma.\) As  \(T^1_R\) is \(\operatorname{Coker}(\beta),\) the previous discussion shows that \(\beta\) is homogeneous of degree \(0\) and \(T^1_R\) is a \(\mathbb{Z}\)--graded \(\mathbb{C}\)--vector space, which in fact is finite-dimensional, as \(R\) has isolated singularity. \footnote{The \(\mathbb{C}\)--algebra \(T^1_R\) is called Tjurina algebra and its dimension as vector space is denoted by \(\tau_R=\dim_\mathbb{C}T^1_R.\)} Therefore, 
\[T^1_R=\bigoplus_{n\in\mathbb{Z}} T^{1}_R(n),\] 
where \(T^1_R(n)\) are the graded pieces of degree \(n.\) 
\medskip

Several properties of \(T^1_R\) and its graded pieces are known, as one can see for example in \cite{Buchweitz,Pinkham1} and the references therein. Also, a full characterization of semigroups with \(T^1_R(n)=0\) for all \(n\geq 0\) is given by Rim and Vitulli in \cite[Theorem 4.7]{RV77}. As to go into the detail could take longer than needed for the purposes of this survey, we will only mention the following theorem which allows to compute the dimensions of the graded pieces of \(T^1_R.\) As in section \ref{subsec:Resolutions}, let \((f_{0,1},\dots,f_{0,l_1})=I_S\) be a minimal set of generators the defining ideal. Each \(f_{0,i}\) comes from  a syzygy \((u_i,v_i)\in\kappa\) in such a way \(f_{0,i}=\mathbf{x}^{u_i}-\mathbf{y}^{v_{i}}.\) Let us denote by \(z_i:=u_i-v_i\) and by \(m_i=\deg(f_{0,i})\) then
\begin{theorem}\cite[Theorem 2.2.1]{Buchweitz}
	\[\dim T^1_R(n)=\max\{0,\; |\; \{a_i\colon\;a_i+n\notin S\}|-1\}-\operatorname{rank}\{z_i\colon\;m_i+n\notin S\}.\]
\end{theorem}

\subsection{Milnor number of monomial curves}
The \(\delta\)--invariant of a reduced curve singularity is one of the most important numerical invariants of the curve. In this case, Buchweitz and Greuel \cite{BG80} also defined the Milnor number, \(\mu,\) which they showed it is closely related to the \(\delta\)--invariant. As the definition of the \(\delta\)--invariant is the same as in the case of a numerical semigroup, let us now recall the definition of Milnor number.
\medskip

Let \((C,0)\subset (\mathbb{C}^{g},0)\) be a reduced curve singularity. Let \(R\) be the local ring of the curve at \(0\) and \(\mathcal{O}=\mathbb{C}\{x_1,\dots,x_g\}\) the ring of convergent power series at \(0.\) As \(R\) is reduced of dimension \(1,\) it has a canonical module. Following \cite{BG80}, let \(\omega_{C,0}\) be the dualizing module or canonical module, which can be defined following Grothendieck definition \cite{LocalCohomology} as  
\[\omega_{C,0}\colon=\operatorname{Ext}^{g-1}_{\mathcal{O}}(R,\Omega^g_{\mathbb{C}^g,0}).\]
As in \cite{BG80}, if \(d:R\rightarrow \Omega^1_R\) denotes the exterior derivation one can consider the maps 
\[dR\rightarrow \Omega^1_R\rightarrow \overline{\Omega}^1_R\rightarrow \omega_{C,0}\]
where \( dR \to \Omega^1_R \) is the inclusion, \( \Omega_R \to \overline{\Omega}^1_R \) is given by the pull-back of forms under the normalization morphism, and \( \overline{\Omega} \to \omega_{C,0} \) is the inclusion. Therefore, the Milnor number of \((C,0)\) is defined as 
\[\mu=\mu(C,0)=\dim_\mathbb{C}\frac{\omega_{C,0}}{d R}.\]
This definition of Milnor number also have topological properties as in the classical case \cite[Section 4]{BG80}. The following formula is extremely useful in the case of monomial curves.
\begin{theorem}\cite[Proposition 1.2.1]{BG80}
	Let \((C,0)\) be a reduced curve singularity. Let \(r\) be the number of branches, let \(\delta\) denotes its \(\delta\)--invariant and \(\mu\) its Milnor number. Then,
	\[\mu=2\delta+r-1.\]
\end{theorem}

Let us now return to the case of \((C,0)\) being the monomial curve defined by a numerical semigroup \(S.\) In this case, by definition, our curve is irreducible, i.e. \(r=1.\) This means that \(\mu=2\delta.\)  As we already mentioned, the \(\delta\)--invariant can be computed from the elements of the semigroup which are smaller than the conductor. In particular, this shows that the Milnor number of a monomial curve can be easily computed. 
\medskip

Finally, recall that a curve is Gorenstein if and only if its semigroup of values is symmetric. Therefore, this is in fact equivalent to the Milnor number being equal to the conductor of the semigroup. This is not surprising, as it is well known that a Cohen-Macaulay ring is Gorenstein if and only if the ring itself is a canonical module (see for example \cite{Bass63,BrunsHerzogbook}) In fact, the Milnor number equals the conductor in more general situations for other classes of complete intersection singularities as showed by Greuel in  \cite{Greuel80}.

\begin{rem}
	Following \cite[Chapter 3]{BrunsHerzogbook} (see also \cite{HKbook71}), a maximal Cohen-Macaulay module \( C \) of type \( 1 \) and of finite injective dimension is referred to as a canonical module of \( R \). This definition coincides with the previous one as showed in \cite{BrunsHerzogbook,HKbook71}. In \cite{Jager77}, J\"{a}hger showed that the canonical module can be explicitly computed from the semigroup of values.
\end{rem}



\section{Characterizations of complete intersection monomial curves}\label{sec:CI}

In short, a complete intersection is a curve whose defining ideal has the smallest possible number of generators. According to the previously given definition of a complete intersection semigroup, this means that, if \(S=\langle a_1,\dots,a_g\rangle\) is generated by \(g\) elements, then we have an exact sequence
\begin{equation}\label{eq:curvemonomialalgebra}
	\mathbb{C}[x_1,\dots,x_g]^{g-1}\rightarrow\mathbb{C}[x_1,\dots,x_g]\xrightarrow{\varphi} \mathbb{C}[C^{S}]\rightarrow 0,
\end{equation}
with \(\operatorname{ker}\varphi=(f_1,\dots,f_{g-1})\) and \(f_1,\dots,f_{g-1}\) define a regular sequence. There exists a vast literature about complete intersection curves due to their excellent geometric and algebraic properties (see for example \cite[2.3]{BrunsHerzogbook} and the references therein).  One of the most useful characterizations of complete intersection rings is the following due to Ferrand \cite[thm. 2]{FerIC67} and Vasconcelos \cite{VasIC68} independently 
\begin{theorem}
	Let \(\mathbb{K}\) be a field and \(A\) a finitely generated \(\mathbb{K}\)--algebra, with quotient field \(\mathbf{k}\) separable over \(\mathbb{K}\). Then \(A\) is locally a complete intersection if and only if the module of K\"{a}hler differentials \(\Omega^1_\mathbb{K}(A)\) has projective dimension less or equal than \(1.\)
\end{theorem}

Observe that, from this point of view, the condition of being a complete intersection curve is purely algebraic and closely related to the properties of K\"{a}hler differentials. Coming back to monomial curves, it is natural to ask up to what extent one can provide a combinatorial characterization of complete intersection numerical semigroups. This is precisely the main aim of this section. 
\medskip

On one hand, the most commonly used characterization of a complete intersection numerical semigroup was given by Delorme \cite{delormeglue}. As we will see, his characterization focuses on the degree of the minimal generators of the defining ideal and the arithmetics of the generators of the semigroup. On the other hand, there is an alternative characterization provided by Herzog and Kunz \cite{HK71} in terms of the conductor of the semigroup and the degrees of the generators of the defining ideal. Due to the fact that \cite{HK71} is originally written in German and delves into intricate technicalities, its impact may not fully reflect its significance.  However, the results, which are applicable to any local Noetherian ring of dimension \(1\), establish intriguing connections between the combinatorics of the value semigroup and the ring's algebraic properties. Since our main focus is on the semigroup algebra associated with a monomial curve defined by a numerical semigroup, we will present the results primarily in this setting. However, we believe that this specific case effectively illustrates the core arguments applicable in the general situation. Our goal is to improve the content's accessibility, potentially fostering wider utilization of these important results.

\subsection{Delorme's characterization}\label{subsec:delorme}
In 1976, Delorme \cite[Lemme~7]{delormeglue} provided an extremely useful combinatorial characterization of numerical semigroups whose semigroup algebra is a complete intersection. Through a ``suite distingu\'ee'', he extracts the combinatorial properties underlying the definition of being a complete intersection. Let \(G\subset \mathbb{N}\setminus\{0\}\) be a finite and non-empty subset. For any subset \(L\subseteq G,\) we denote \(S_L=\langle a\in L\rangle\) the sub-monoid of \(\mathbb{N}\) generated by \(L,\) \(A(S_L)=\bigoplus_{\alpha\in S_L}\mathbb{C}t^\alpha\) the semigroup algebra associated to \(S_L\) and \(M(S_L)\) the affine graded \(\mathbb{C}\)--algebra associated to \(S_L,\) i.e. if \(L=\{a_1,\dots,a_l\}\) then \(M(S_L)\simeq\mathbb{C}[x_1,\dots,x_l]\) with \(\deg(x_i)=a_i.\)
\begin{defin}\label{defin:distseq}
	Let \(G\subset \mathbb{N}\setminus\{0\}\) be a finite and non-empty subset with \(g=|G|.\)	A distinguished sequence (suite distingu\'ee)  over \(G\) is a pair \((\mathcal{P},\mathcal{Z})\) where 
	\begin{enumerate}
		\item [\(\bullet\)] \(\mathcal{P}=(P_1,\dots,P_g)\) is a sequence of partitions of \(G\) into non-empty sets defined recursively as follows: 
		\begin{enumerate}
			\item [$\diamond$] \(P_g=\{\{a_i\}\ |\ 1\leq i\leq g\},\)
			\item [$\diamond$] for each \(1\leq i\leq g-1,\) there are \(L_i,L'_{i}\in P_i\) such that \(L_i\neq L'_i\) and \[P_{i-1}=(P_i\setminus\{L_i,L'_i\})\cup(\{a\in L_i\cup L'_i\}).\]
		\end{enumerate}
		\item [\(\bullet\)] \(\mathcal{Z}=(Z_2,\dots,Z_g)\subset M(S_G)\) is a sequence of monomials associated to \(\mathcal{P}\) such that if \(L_i\) and \(L'_i\) are the sets used to obtain \(P_{i-1}\) from \(P_i\) then 
		\[Z_i=\mathbf{x}-\mathbf{y}\quad\text{with}\quad \mathbf{x}\in M(S_{L_i}),\ \mathbf{y}\in M(S_{L'_i})\quad \text{and}\quad \deg(\mathbf{x})=\deg(\mathbf{y}).\]
	\end{enumerate}
\end{defin}

After establishing this definition, it is important to highlight a relevant consideration. Observe that a distinguished sequence is nothing but a special representation of \(g-1\) syzygies. It codifies through \(\mathcal{P}\) the congruences that could give rise to a presentation of the semigroup and through \(\mathcal{Z}\) some generators of the defining ideal of the curve. Therefore, it is trivial to see that \(S(G)\) is a complete intersection if and only if there exists a distinguished sequence \((\mathcal{P},\mathcal{Z})\) such that \(I(G)\) is generated by \(\mathcal{Z}\) \cite[Lemme 6]{delormeglue}. In fact, this idea allows to easily visualize a Theorem of Herzog \cite[sec. III]{herzog} which states that a numerical semigroup with three generators is not a complete intersection if the generators are pairwise coprime.
\medskip

The full power of the distinguished sequence comes with the following combinatorial characterization of a complete intersection. 

\begin{proposition}\cite[Lemme 7]{delormeglue}\label{prop:DelormeLema7}
	Let \((\mathcal{P},\mathcal{Z})\) be a distinguished sequence over \(G\) with \(g\geq 2.\) Let us denote \(d_1:=\gcd(L_2),\) \(d_2:=\gcd(L'_2),\) \(m:=\lcm(d_1,d_2)\) and by \(r=\deg(Z_2).\) Then, we have the following properties:
	\begin{enumerate}
		\item Let \(F, F'\) be minimal generators of \(I(L_2)\) and \(I(L'_2)\) respectively. Then, \(F,F',Z_2\) is a minimal system of generators of \((I(L_2),I(L'_2),Z_2).\)
		\item \((I(L_2),I(L'_2),Z_2)\) generates \(I(G)\) if and only if \(m=r.\)
	\end{enumerate}
\end{proposition}
\begin{proof}
	The detailed proof can be found in \cite[Lemme 7]{delormeglue}. However, we will revisit in detail the proof of part \((2),\) as it is the key result used to obtain the combinatorial characterization of complete intersection numerical semigroups. 
	\medskip
	
	Let us first assume \(m=r.\) Thanks to Herzog \cite[Chap I]{herzog} (see also Theorem \ref{thm:Herzogpresentation}), we know that \(I(G)\) is a binomial ideal which can be generated through all the elements of the form \(\mathbf{x}_1\mathbf{y}_1-\mathbf{x}_2\mathbf{y}_2\) with \(\mathbf{x}_i\in M(S_{L_2}),\) \(\mathbf{y}_i\in M(S_{L'_2})\) and if we denote by \(\alpha_i=\deg \mathbf{x}_i\) and \(\beta_i=\deg \mathbf{y}_i\) then \(\alpha_1+\beta_1=\alpha_2+\beta_2.\) In this way, as \(\alpha_i\in S_{L_2}\) and \(\beta_i\in S_{L'_2},\) it is obvious that \(\alpha_1-\alpha_2=\beta_2-\beta_1=km\) is a certain multiple of \(m=\lcm(d_1,d_2).\) 
	\medskip
	
	As in the definition of distinguished sequence, let us denote \(Z_2=\mathbf{x}-\mathbf{y}\) with \(\mathbf{x}\in M(S_{L_2}),\) \(\mathbf{y}\in M(S_{L'_2})\) and by hypothesis \(\deg(\mathbf{x})=\deg(\mathbf{y})=m.\) Given any element of the form \(\mathbf{x}_1\mathbf{y}_1-\mathbf{x}_2\mathbf{y}_2,\) we can then rewrite it as 
	\[\mathbf{x}_1\mathbf{y}_1-\mathbf{x}_2\mathbf{y}_2=\mathbf{y}_1(\mathbf{x}_1-\mathbf{x}^k\mathbf{x}_2)-\mathbf{x}_2(\mathbf{y}_2-(-\mathbf{y})^k\mathbf{y}_1)+\mathbf{y}_1\mathbf{x}_2(\mathbf{x}^k-(-\mathbf{y})^k).\]
	Obviously, \(\mathbf{x}^k-(-\mathbf{y})^k=Z_2^k-\sum_{l=1}^{k-1}\binom{k}{l}(-1)^l\mathbf{x}^{k-l}\mathbf{y}^l.\) Thus, after the corresponding substitutions we have \[\mathbf{x}_1\mathbf{y}_1-\mathbf{x}_2\mathbf{y}_2\in (I(L_2),I(L'_2),Z_2).\]
	
	Let us now prove the converse. Suppose that \(r>m.\) Obviously we have \(r=km.\) As \(S_{L_2}/d_1\) and \(S_{L'_2}/d_2\) are numerical semigroups, then there exists \(c_1,c_2\in\mathbb{N}\) such that we have \(c_1+\mathbb{N}\subset S_{L_2}/d_1 \) and \(c_2+\mathbb{N}\subset S_{L'_2}/d_2.\) Therefore, there exists \(c\in\mathbb{N}\) such that \(c+m\mathbb{N}\subset S_{L_2}\cap S_{L'_2}.\) In particular, there exists \(q\in S_{L_2}\cap S_{L'_2} \)such that \(q\notin r\mathbb{N}.\) This combinatorial property translates in the existence of two monomials \(\mathbf{x}_1\in M(S_{L_2})\)  and \(\mathbf{y}_1\in M(S_{L'_2})\) such that \(\deg(\mathbf{x}_1)=\deg(\mathbf{y}_1)=q\) and \(Z:=\mathbf{x}_1-\mathbf{y}_1\in I(G).\) However, the class of \(Z\) in \((M(S_{L_2})\otimes M(S_{L'_2}))/(Z_2)\) is not zero and thus provides a contradiction with the fact that \(I(G)=(I(L_2),I(L'_2),Z_2).\) 
\end{proof}

Proposition \ref{prop:DelormeLema7} shows the strong connection between the degrees of the syzygies and the fact that those syzygies generate a regular sequence.

\begin{corollary}\label{cor:syzCI}
	Let \((\mathcal{P},\mathcal{Z})\) be a distinguished sequence over \(G\). \(\mathcal{Z}\) generates \(I(G)\) if and only if the following condition is satisfied
	\begin{equation}
		\text{for all}\quad 2\leq i\leq g,\quad \deg(Z_i)=\lcm(\gcd(L_i),\gcd(L'_i)).
	\end{equation}
	
	In particular, \(A(S_G)\) is a complete intersection if and only if there exists a distinguished sequence \((\mathcal{P},\mathcal{Z})\) satisfying the previous conditions.
\end{corollary}

Delorme's characterization can be thought of as a decomposition theorem of the semigroup algebra.  Set $A:=\{a_1,\dots,a_g\}$; for \(S=\langle A\rangle,\) Delorme's characterization (Corollary \ref{cor:syzCI}) states that \(S\) is a complete intersection numerical semigroup if and only if there exists a partition \(A=A_1\sqcup A_2\) of the set of generators with $A_1 \neq \emptyset \neq A_2$ such that \(\mathbb{C}[S_{A_i/d_i}]\) are complete intersections defined by \(I_i:=\ker\varphi_i\), where \(d_i:=\gcd(A_i)\) for $i=1,2$, and \(\mathbb{C}[S]\) is defined by \(I_1+I_2+\rho\) with \(\deg(\rho)=\operatorname{lcm}(d_1,d_2).\) More precisely, if we set \(g_i:=|A_i|\), then \(g_1+g_2=g\) and we have the exact sequences
\begin{equation*}
	\begin{split}
		\mathbb{C}[x_1,\dots,x_{g_1}]^{g_1-1}\rightarrow\mathbb{C}[x_1,\dots,x_{g_1}]\xrightarrow{\varphi_1} \frac{\mathbb{C}[x_1,\dots,x_{g_1}]}{I_1}\simeq \mathbb{C}[S_{A_1/d_1}]\rightarrow0\\
		\mathbb{C}[y_{1},\dots,y_{g_2}]^{g_2-1}\rightarrow\mathbb{C}[y_{1},\dots,y_{g_2}]\xrightarrow{\varphi_2} \frac{\mathbb{C}[y_1,\dots,y_{g_2}]}{I_2}\simeq \mathbb{C}[S_{A_2/d_2}]\rightarrow 0
	\end{split}
\end{equation*}

In addition, one needs to define the binomial \(\rho\) in separated variables with degree \(\lcm(d_1,d_2).\) Thus, Delorme's results states that there is a natural decomposition of the semigroup algebra of \(S\) in the form
\[\mathbb{C}[S]=\frac{\mathbb{C}[S_{A_1/d_1}]\otimes \mathbb{C}[S_{A_2/d_2}]}{\rho }.\]

\begin{rem}
	The set of generators \(A\) of \(S\) does not need to be a minimal generating set. 
\end{rem}
Moreover, based on Proposition \ref{prop:DelormeLema7} and the definition of a distinguished sequence (Definition \ref{defin:distseq}), Delorme showed the following algorithm to check if a numerical semigroup is a complete intersection or not \cite[sec. 14]{delormeglue}. Before to introduce it, let us fix some notation. Let \(G=\{a_1,\dots,a_g\}\) denotes the set of generators of a numerical semigroup \(S.\) For \(L\subsetneq G,\) recall that \(S_L\) is the semigroup generated by \(L\) and define 
\[m_L:=\min(S_L\cap S_{G\setminus L} ).\]
The aim of the algorithm is to construct a distinguished sequence over \(G\) satisfying the characterization of Proposition \ref{prop:DelormeLema7}. 
\begin{itemize}
	\item The initial input is \(P_g=\{\{a_i\}\ |\ 1\leq i\leq g\}.\)
	\item Assuming that for some \(2\leq k\leq g\) the partition \(P_k\) has been already constructed, for all sets \(L\in P_k\) compute \(m_L.\) Then
	\begin{itemize}
		\item If \(m_L\neq m_{L'}\) for all \(L,L'\in P_k\) then \(S\) is not a complete intersection.
		\item If there exist \(L,L'\in P_k\) with \(L\neq L'\) such that \(m_L=m_{L'}\) then 
		\begin{itemize}
			\item If \(m_L\neq \lcm(\gcd(L),\gcd(L'))\) then \(S\) is not a complete intersection.
			\item If \(m_L= \lcm(\gcd(L),\gcd(L'))\) then we set \(L_k=L\) and \(L'_{k}=L'\) and define \(P_{k-1}\) as in Definition \ref{defin:distseq}. 
		\end{itemize}
	\end{itemize}
	
\end{itemize}

\begin{ex}\label{ex:exDelormealgorithm1}
	Let us consider the value semigroup of Example \ref{ex:CIringnotSemi} \(S=\langle 6,8,10,17,19\rangle.\) We have \(P_5=\{\{6\},\{8\},\{10\},\{17\},\{19\}\}\) and 
	\[\begin{array}{c}
		m_6=18,\quad m_8=16,\quad m_{10}=20,\quad m_{17}=34,\quad m_{19}=38.
	\end{array}\]
	As \(m_L\neq m_{L'}\) for all \(L,L'\in P_5,\) Delorme's algorithm states that \(S\) is not a complete intersection semigroup.
\end{ex}
\begin{ex}\label{ex:exDelormealgorithm2}
	Let \(S=\langle 15, 16, 24, 28\rangle,\) we have \(P_4=\{\{15\},\{16\},\{24\},\{28\}\}\) and \(m_{16}=m_{24}=48,\) \(m_{15}=60\) and \(m_{28}=56.\) As \(m_{16}=\lcm(16,24)\) then we take \(L_3=\{16\}\) and \(L'_3=\{24\}\) and define \(P_3=\{\{16,24\},\{15\},\{28\}\}.\) We compute now the new possible degrees for a syzygy: 
	\(m_{\{16,24\}}=56=m_{28}\) and \(m_{15}=60.\) Observe that
	\[m_{\{16,24\}}=m_{28}=56=\lcm(\gcd(16,24),\gcd(28))=\lcm(8,28).\]
	Therefore, we define \(P_2=\{\{16,24,28\},\{15\}\}.\) The only element to evaluate is \(m_{15}=60=\lcm(\gcd(16,24,28),15)=\lcm(4,15).\) From this, we conclude that \(S\) is a complete intersection.
	
	Let us finally compute the equations of \(C^S,\) again with the help of this algorithm. Observe that we have already computed the partitions of a distinguished sequence, then it only remains to compute the syzygies. In order to do that, let us consider the ring \(\mathbb{C}[x_1,x_2,x_3,x_4]\) with \(\deg(x_1)=15,\) \(\deg(x_2)=16,\) \(\deg(x_3)=24\) and \(\deg(x_4)=28.\) We need to express the relevant \(m_{L}\) in terms of the corresponding partition as follows:
	\[\begin{array}{ccc}
		m_{16}=m_{24}=48=3\cdot16=2\cdot24&\rightsquigarrow& x_2^3-x_3^2=0\\
		m_{\{16,24\}}=m_{28}=56=16\cdot2+24=28\cdot2&\rightsquigarrow& x_2^2x_3-x_4=0\\
		m_{15}=m_{\{16,24,28\}}=60=4\cdot 15=2\cdot 16+28&\rightsquigarrow&x_2^2x_4-x_1^4=0.
	\end{array}\]
	In fact observe that \(S\) is a free numerical semigroup when we consider the following order \(\deg(x_1)=16,\) \(\deg(x_2)=24,\) \(\deg(x_3)=28\) and \(\deg(x_4)=15\) for the generators of \(S.\) Under the order of generators \(a_1=16,a_2=24,a_3=28\) and \(a_4=15\), it is easy to check the free condition \eqref{eq:freecond}.
\end{ex}
\begin{rem}
	It is worth noting that Bermejo, García-Marco and Salazar-González \cite{BGS07} has proposed an alternative algorithm to decide whether the defining ideal is a complete intersection or not. 
\end{rem}

Finally, we must mention some important generalizations of Delorme's characterization of a complete intersection numerical semigroup to the case of affine semigroups. An affine semigroup is a finitely generated submonoid of \(\mathbb{N}^n\). For affine semigroups the definition of complete intersection is analogous to the case of \(n=1.\) In this more general context, Rosales \cite{RosGluing} and independently Fischer and Shapiro \cite{FMAffine} provided a characterization similar to Delorme's for a complete intersection affine semigroup. More concretely, they independently established, using different methods, that an affine semigroup is a complete intersection if and only if it can be split into two "smaller" complete intersection affine semigroups. In addition, Rosales \cite{RosGluing} proposed to call gluing to this splitting operation; we will review this construction in Section \ref{subsec:poincaresemigroup}. However, he did not considered the algebraic restrictions prescribed by Auslander-Buchsbaum formula for this operation that naturally arise in the minimal resolution when \(n>1.\) While the combinatorial potential of the gluing of semigroups has been largely developed, it was only recently that the algebraic conditions defining a semigroup in \(\mathbb{N}^g\) as a gluing of other semigroups were articulated. These insights have been systematically explored by Giménez and Srinivasan in \cite{Philippe1,Philippe2,Philippe3}, where they provide the necessary and sufficient algebraic conditions for an affine semigroup to be obtained by gluing.

\subsection{Herzog-Kunz characterization}
In 1971 Herzog and Kunz \cite{HK71} provided a characterization of one dimensional complete intersection local rings in terms of the Noether, Dedekind and K\"{a}hler differents. As a consequence, they provided a quite interesting characterization of complete intersection numerical semigroups in terms of its conductor. To revisit their result, we will first briefly recall the basic properties of K\"ahler and Dedekind differents which are the ones needed to show their results. For a more detailed discussion about these differents and their relations we refer to \cite{Berger60, Berger63,Kunz68} and Kunz's book \cite[\S 10 and Appendix G]{Kunzbook}.
\medskip

Following \cite[\S 4]{HK71}, let \(R\) be a one dimensional complete local ring, let assume for simplicity that \(\mathbb{C}\) is its residual field and \(\mathfrak{m}\) its maximal ideal. Let \(L\) be the field of fractions of \(R\) and \(K\) be the field of fractions of \(\mathbb{C}[x].\) From the trace map \(\sigma_{L/K}: L\rightarrow K\) (see also \cite[Appendix F]{Kunzbook}), the Dedekind different \(\mathcal{D}_D(R/\mathbb{C}[x])\) is defined through the complementary module 
\[\mathcal{L}(R/\mathbb{C}[x])=\{y\in L\;|\;\sigma_{L/K}(yR)\subseteq\mathbb{C}[x]\}.\]
The Dedekind different is defined as 
\[\mathcal{D}_D(R/\mathbb{C}[x])=\{z\in L\;|\;z\mathcal{L}(R/\mathbb{C}[x])\subseteq R\}.\]

We are interested in the case where \(R=\mathbb{C}[S]\) and \(S=\langle a_1,\dots,a_g\rangle\) is a numerical semigroup. In this case, in order to take an element of the maximal ideal it is enough to consider \(t^h\) with \(h\in S\setminus\{0\}.\) Observe that in this context \(L=\mathbb{C}(t)\) and \(K=\mathbb{C}(t^h),\) thus any element \(z\in L\) can be written as 
\[z=\frac{1}{ht^{h-1}}\left(\alpha_0+\alpha_1t+\cdots+\alpha_{h-1}t^{h-1}\right),\quad \alpha_i\in K\]
in such a way that \(\sigma_{L/K}(z)=\alpha_{h-1}.\) Therefore \cite[\S 4 eqn. (2), pg. 30/54]{HK71}, the complementary module is 
\[\mathcal{L}(R/\mathbb{C}[t^h])=\bigoplus_{\gamma\in \Ap(S,h)}\mathbb{C}[t^h]\cdot t^{-\gamma}.\]

For an element \(h\in S\), in order to determine the graded pieces of the Dedekind different one needs to define the following \(S\)--semimodule
\[D(S,h):=\bigcap_{\gamma\in \Ap(S,h)}(\gamma)\] 
Thus, it is not difficult to prove \cite[Proposition 4.3]{HK71} that the Dedekind different can be expressed as 
\[\mathcal{D}_D(R/\mathbb{C}[t^h])=\bigoplus_{\gamma\in D(S,h)}\mathbb{C}\cdot t^{\gamma}.\]
\begin{rem}
	In general, the ring \(R\) is not necessarily a graded one. Therefore, as already mentioned the ring \(\operatorname{gr}(R)\) may contain less information. Let \(v:R\rightarrow\mathbb{Z}\) be the discrete valuation in the local ring \(R\) induced by its normalization. It happens that \(v(\mathcal{D}_D(R/\mathbb{C}[x]))\subseteq v(\mathcal{D}_D(\operatorname{gr}(R)/\mathbb{C}[x]))\). This means that when our curve is not monomial then we can loose some information when operating with their semigroup of values. Most of the results that we will explain are still valid in this more general situation as one can see in \cite{HK71}.
\end{rem}

At this point, Herzog and Kunz \cite[Proposition 4.4 and 4.6]{HK71} show the following characterization of a symmetric semigroup \cite[Proposition 4.4]{HK71}
\begin{proposition}\cite[Proposition 4.4]{HK71} For any numerical semigroup we have 
	\[D(S,h)\subseteq (h+c-1)\]
	with equality if and only if \(S\) is symmetric. Moreover, if \(S\) is not symmetric then \(D(S,h)\) is not a principal \(S\)--semimodule.
	
	In terms of Dedekind differents, if \(v:R\rightarrow\mathbb{Z}\) is the discrete valuation of \(R=\mathbb{C}[S]\)  we have \(v(\mathcal{D}_D(R/\mathbb{C}[x]))=D(S,h).\) In particular, this means that the Dedekind different is a principal ideal if and only if \(S\) is symmetric. 
\end{proposition}
\begin{proof}
	Since \(h+c-1\in \Ap(S,h)\) then \(D(S,h)\subseteq (h+c-1).\) Now, let us characterize the equality. 
	
	Assume first that \(S\) is symmetric. Let \(\gamma\in \Ap(S,h),\) this is equivalent to \(\gamma-h\notin S.\) As \(S\) is symmetric we have \(c-1-(\gamma-h)\in S.\) Thus \(c-1+h\in (\gamma).\) Therefore, we have \(c-1-h\in D(S,h)\) and we obtain the required equality.
	
	Assume now that \(S\) is not symmetric. Then there exists \(x\in S\) with \(x\in\{1,\dots,c-1\}\) such that \(c-1-x\in S.\) On the other hand, we can always write \(x=kh+\gamma\) with \(\gamma\in \Ap(S,h).\) Therefore,  for any \(-k\geq 1\) we have \(c-1-kh\notin (\gamma).\) In particular, this is satisfied by \(k=-1\) and thus \(c-1+h\notin (\gamma)\) and then \(D(S,h)\subseteq(h+c-1)\setminus\{h+c-1\}.\) Therefore, the inclusion is strict. From here, we have that equality holds only if \(S\) is symmetric.
	
	Finally, to prove that \(D(S,h)\) is not principal in the case where \(S\) is not symmetric is an easy consequence of the fact that any \(\gamma\geq h+2c-1\) satisfies \(\gamma\in D(S,h).\) We refer to \cite[Lemma 4.5]{HK71} for the conclusion.
\end{proof}

Now we will discuss the main properties about K\"{a}hler differentials. As \(R=\mathbb{C}[S]\) is a one dimensional complete local ring, we can always consider \(y\in\mathfrak{m}\) as a local parameter; from which we obtain an injective morphism \(\mathbb{C}[y]\rightarrow \mathbb{C}[S].\) In section \ref{subsec:deformations} we defined the module of K\"{a}hler differentials \(\Omega^1_{R},\) hence we have the following exact sequence 
\[0\rightarrow Rdy\rightarrow \Omega^1_{R}\rightarrow\Omega^1_{R/\mathbb{C}[y]}\rightarrow 0\]
where \(dy=\sum_{i=1}^{g}r_jdt^{a_j}\in \Omega^1_{R}.\) We are interested in the \(0\)--fitting ideal of the module \(\Omega^1_{R/\mathbb{C}[y]}.\) Observe that a matrix presentation of \(\Omega^1_{R/\mathbb{C}[y]}\) is given by 

\[\left(\begin{matrix}
	r_1 & \dots & r_g \\ 
	\frac{\partial f_{0,1}}{\partial t^{a_1}} & \dots & \frac{\partial f_{0,1}}{\partial t^{a_g} }\\ 
	\vdots &   & \vdots \\ 
	\frac{\partial f_{0,l_1}}{\partial t^{a_1}} & \dots & \frac{\partial f_{0,l_1}}{\partial t^{a_g} } \notag
\end{matrix}\right).\]

Therefore, as \(l_1\geq g-1\) the \(0\)--fitting ideal of \(\Omega^1_{R/\mathbb{C}[y]}\) is generated by the minors of order \(g\) of this matrix. This fitting ideal is in fact an \(R\)--module called K\"{a}hler different and denoted by \(\mathcal{D}_K(\mathbb{C}[S]/\mathbb{C}[t^{a_1}])\). As an application of Ferrand and Vasconcelos Theorem, Herzog and Kunz obtain the following result: 
\begin{theorem}\cite[Satz 5.1]{HK71}
	Under the previous notations and assumptions, the following statements are equivalent:
	\begin{enumerate}
		\item  R is a complete intersection.
		\item \(\mathcal{D}_K(\mathbb{C}[S]/\mathbb{C}[t^{a_1}])\) is a principal ideal.
	\end{enumerate}
\end{theorem}

\noindent Moreover, Berger in 1960 \cite{Berger60} (see also \cite{Kunz68}) proved the following inclusion
\[\mathcal{D}_K(\mathbb{C}[S]/\mathbb{C}[t^{a_1}])\subseteq \mathcal{D}_D(\mathbb{C}[S]/\mathbb{C}[t^{a_1}])\]
which allows to provide the following
\begin{corollary}\cite[Korollar 5.2]{HK71}
	\(R\) is a complete intersection if and only if \(R\) is Gorenstein and for every \(x\in\mathfrak{m}\) the equality \(\mathcal{D}_K(R/\mathbb{C}[x])= \mathcal{D}_D(R/\mathbb{C}[x])\) holds.
\end{corollary}

The purpose is now to use the previous characterizations to revisit Herzog and Kunz characterization \cite{HK71} of the conductor of a complete intersection semigroup. The key point is to focus in the \(S\)--ideal defined by \(v(\mathcal{D}_K(\mathbb{C}[S]/\mathbb{C}[t^{a_1}])).\) 
\medskip

To understand the K\"{a}hler different \(\mathcal{D}_K(\mathbb{C}[S]/\mathbb{C}[t^{a_1}])\) observe that we need to analyze the order \(g\) minors of the matrix presentation of \(\Omega^1_{R/\mathbb{C}[y]}.\) To do so, let us recall that in section \ref{sec:monomial}, we defined the set of congruences \(\kappa\) of a numerical semigroup. With the notations of section \ref{sec:monomial}, recall that \(\kappa\) defines naturally the kernel of the natural epimorphism \(\rho:\mathbb{N}^g\rightarrow S.\) A syzygy \((u,v)\in\kappa\) gives rise to a non-trivial syzygy of the form \(\mathbf{x}^u-\mathbf{y}^{v}.\) Moreover, as explained, the ideal \(I_S\) is defined through the set of congruences (see Theorem \ref{thm:Herzogpresentation}). Let us then consider a subset \(\{(u_1,v_1),\dots,(u_{g-1},v_{g-1})\}\subset\kappa\) of congruences of cardinality equal to \(g-1.\) Let us denote by \(z_i:=u_i-v_i,\) then it is easy to check that 
\[\frac{\partial F_{(u_i,v_i)}}{\partial t^{a_j}}=z_j\cdot t^{m_i-a_j}\]
with \(m_i=\deg F_{(u_i,v_i)}. \) The point is precisely to characterize in terms of their degrees when the minimal set of generators of \(I_S\) has cardinality equal to \(g-1.\) Recall now that \(u,v\in\mathbb{N}^g\) so each \(z_j\in\mathbb{N}^g\) is of the form \(z_j=(z_{j,1},\dots,z_{j,g}).\) This means that the K\"{a}hler different is generated by all the possible minors of the form

\[\begin{vmatrix}
	1 & \dots & 0 \\ 
	\frac{\partial F_{(u_1,v_1)}}{\partial t^{a_1}} & \dots & \frac{\partial F_{(u_1,v_1)}}{\partial t^{a_g} }\\ 
	\vdots &   & \vdots \\ 
	\frac{\partial F_{(u_{g-1},v_{g-1})}}{\partial t^{a_1}} & \dots & \frac{\partial F_{(u_{g-1},v_{g-1})}}{\partial t^{a_g} } \notag
\end{vmatrix}=t^{\sum_{i=1}^{g-1}m_i-\sum_{j=2}^{g}a_j}\cdot\begin{vmatrix}
	z_{1,2} & \dots & z_{1,g}\\ 
	\vdots &   & \vdots \\ 
	z_{g-1,2} & \dots & z_{g-1,g} \notag
\end{vmatrix}.\]
\medskip

In this way, this determinant equals \(0\) if and only if \(z_1,\dots,z_{g-1}\) are \(\mathbb{Z}\)--linearly dependent. Then, similarly to the case of the value set of the Dedekind different, it is natural to define the following \(S\)--semimodule: let \(\mathcal{P}_{g-1}\) be the subset of \(\kappa\) defined by all the possible sets \(\{(u_1,v_1),\dots,(u_{g-1},v_{g-1})\}\subset \kappa\) of cardinality \(g-1\) such that the corresponding set \(\{z_1,\dots,z_{g-1}\}\) is \(\mathbb{Z}\)--linearly independent. Therefore, we define the \(S\)--semimodule 
\[K(S/a_1):=\bigcup_{ \{(u_1,v_1),\dots,(u_{g-1},v_{g-1})\}\subset \mathcal{P}_{g-1}}\left(\sum_{i=1}^{g-1}(\deg F_{(u_i,v_i)}-a_{i+1})\right)\]
From the previous discussion, it is easy to show \cite[Satz 5.8]{HK71} that 
\[v(\mathcal{D}_K(\mathbb{C}[S]/\mathbb{C}[t^{a_1}]))=K(S/a_1).\]
\begin{rem}
	The set \(K(S/a_1)\) is independent of the system of generators of \(S\) \cite[Korollar 5.9]{HK71}.
\end{rem}

Following Herzog-Kunz \cite[\S 5]{HK71}, let us consider the minimum of all the sums of degrees of syzygies appearing as generators of \(K(S/a_1),\) i.e. 

\[
m(S):=\min \left\{ m_1+\cdots + m_{g-1} : \begin{array}{ccc}
	\ m_i=\deg F_{(u_i,v_i)}, \hspace{-11pt} & (u_i,v_i)\in\kappa,  \ \  z_{m_i}=u_i-v_i \ \  \mbox{and}&\\
	z_{m_1},\ldots , z_{m_{g-1}}&  \mbox{are linearly independent in }\ \mathbb{Z}^g&
\end{array}\right\}
\]

From this definition, Herzog and Kunz proved in \cite[Satz 5.10]{HK71} the following characterization of a complete intersection numerical semigroup.

\begin{theorem}[Kunz-Herzog]\cite[Satz 5.10]{HK71}\label{thm:KHCI}
	Let $S=\langle a_1,\ldots ,a_g\rangle$ be a numerical semigroup with conductor $c(S)$, then:
	\begin{enumerate}
		\item $c(S)\leq m(S)-\sum_{i=1}^g a_i+1$.
		\item The equality holds if and only if $S$ is complete intersection.
	\end{enumerate}
\end{theorem}

\begin{proof}
	Let us denote \(m:=m(S)\) and \(c:=c(S).\) Part \((a)\) follows from the definitions and the previous discussion as
	
	\begin{align*}
		m - \sum_{i=2}^{g} a_i = \min K(S/a_1) &\in v(\mathcal{D}_K(\mathbb{C}[S]/\mathbb{C}[t^{a_1}])) =  K(S/\hspace{-9pt} \raisebox{-18pt}{\rotatebox{90}{\(\supseteq\)}} a_1)\\  
		(a_1 + c - 1) &\supseteq D(S, a_1) = v(\mathcal{D}_D(\mathbb{C}[S]/\mathbb{C}[t^{a_1}])).
	\end{align*}
	Therefore, \(a_1+c-1=\min (a_1+c-1)\leq \min K(S/a_1)=m-\sum_{i=2}^{g}a_i.\)
	
	To prove \((b),\) let us assume first that \(S\) is a complete intersection. In this case, we have the following equalities:
	\begin{equation*}
		\begin{split}
			(m-\sum_{i=2}^{g}a_i)=&K(S/a_1)=v(\mathcal{D}_K(\mathbb{C}[S]/\mathbb{C}[t^{a_1}]))=\\
			=&v(\mathcal{D}_D(\mathbb{C}[S]/\mathbb{C}[t^{a_1}]))=D(S,a_1)=(a_1+c-1)
		\end{split}
	\end{equation*}
	from which we obtain \(c=m-\sum_{i=1}^{g}a_i+1.\)
	
	Assume now that \(S\) is not a complete intersection, but it is symmetric. In this case 
	\[v(\mathcal{D}_K(\mathbb{C}[S]/\mathbb{C}[t^{a_1}]))\subsetneq v(\mathcal{D}_D(\mathbb{C}[S]/\mathbb{C}[t^{a_1}]))=D(S,a_1)=(a_1+c-1)\]
	from which 
	\[a_1+c-1=\min (a_1+c-1)< \min K(S/a_1)=m-\sum_{i=2}^{g}a_i.\]
	Finally, if \(S\) is not symmetric then 
	\[v(\mathcal{D}_D(\mathbb{C}[S]/\mathbb{C}[t^{a_1}]))\subsetneq(a_1+c-1)\]
	from which again we deduce the strict inequality.
	
\end{proof}

\begin{ex}
	Let us consider the semigroup \(S=\langle 5,7,9\rangle\) of the Example \ref{ex:simplComplex}. As we showed in Example \ref{ex:simplComplex}, the degrees of the syzygies associated to the generators of \(I_S\) are \(m_1=14,\) \(m_2=25\) and \(m_3=27.\) Through their factorization in \(S\) we can compute the corresponding vectors \(z_i:\)
	\[\begin{array}{ccc}
		m_1=14=2\cdot 7=5+9&\rightsquigarrow&z_1=(1,-2,-1)\\
		m_2=25=5\cdot 5=7+2\cdot 9&\rightsquigarrow&z_2=(5,-1,-2)\\
		m_3=27=3\cdot 9=4\cdot 5+7&\rightsquigarrow&z_3=(4,1,-3).
	\end{array}\]
	Observe that \(z_3=z_2-z_1.\) We thus have \(m(S)=14+25=39.\) On the other hand, we have \(c(S)=14<39-5-7-9+1=19.\) Thus, \(S\) is not a complete intersection as already expected since the generators of \(S\) are pairwise coprime.
\end{ex}

\begin{ex}
	Let us consider the semigroup \(S=\langle 15,16,24,28\rangle\) of Example \ref{ex:exDelormealgorithm2}. In the process of computing the distinguished sequence, we already computed in Example \ref{ex:exDelormealgorithm2} the syzygies associated to the generators of \(I_S.\) Then, the corresponding vectors \(z_i\) are 
	\[\begin{array}{ccc}
		m_1=48=3\cdot16=2\cdot24&\rightsquigarrow& z_1=(0,3,-2,0)\\
		m_2=56=16\cdot2+24=28\cdot2&\rightsquigarrow& z_2=(0,2,1,-2)\\
		m_3=60=4\cdot 15=2\cdot 16+28&\rightsquigarrow&z_3=(4,-2,0,-1).
	\end{array}\]
	Obviously, they are \(\mathbb{Z}\)--linearly independent, we have \(m(S)=48+56+60=164\) and, as expected,
	\[c(S)=82=164-15-16-24-28+1.\]
\end{ex}

\subsection{Some questions related with simplicial complexes}
Let \(S=\langle a_1,\dots, a_g\rangle\) be a numerical semigroup. Recall that for each \(m\in S\) we have the associated simplicial complex \(\Delta_m\) as defined in section \ref{subsec:Resolutions}. By construction, \(\Delta_m\) is the full simplex in \(g\) vertices for any \(m> c-1+\sum_{i=1}^g a_i.\) Therefore, for any \(m> c-1+\sum_{i=1}^ga_i\) we have \(\tilde{h}_i(\Delta_m)=0\) for all \(i.\) If we denote by \(\theta=|\{\alpha\in S:\; \alpha\leq c-1+\sum_{i=1}^{g}a_i\}|,\) the square of integers $\tilde{h}_i(\Delta_m)$ considered in \cite{CM} is in fact the Betti diagram associated to $R$. 
\begin{table}[H]
	\begin{center}
		\begin{tabular}{c|c c c c c c c} 
			\phantom{m} & $\beta_0$ & $\beta_1$ & $\beta_2$ & $\cdots$  &$\beta_i$ & $\cdots$ & $\beta_{g-1}$  \\[2.5pt]
			\hline\\[-8pt]
			$m_0$ &  $\beta_{0,m_0}$ & $\beta_{1,m_0}$ & $\beta_{2,m_0}$ & $\cdots$ & $\beta_{i,m_0}$ & $\cdots$ & $\beta_{g-1,m_0}$ \\[1.9pt]
			$m_1$ &  $\beta_{0,m_1}$ & $\beta_{1,m_1}$ & $\beta_{2,m_1}$ & $\cdots $& $\beta_{i,m_1}$ & $\cdots$ & $\beta_{g-1,m_1}$ \\[1.9pt]
			$m_2$ &  $\beta_{0,m_2}$ & $\beta_{1,m_2}$ & $\beta_{2,m_2}$ & $\cdots$ & $\beta_{i,m_2}$ &  $\cdots $ & $\beta_{g-1,m_2}$\\[0.9pt]
			$\vdots$ &  $\vdots$ & $\vdots$ &$\vdots$  & $\vdots$ & $\vdots$ & $\vdots$\\[0.9pt]
			$\theta$ &  $\beta_{0,\theta}$ & $\beta_{1,\theta}$ & $\beta_{2,\theta}$ & $\cdots$ & $\beta_{i,\theta}$ &  $\cdots $ & $\beta_{g-1,\theta}$\\[15pt]
		\end{tabular}
		\caption{Betti diagram of a numerical semigroup ring  $\mathbb{C}[S]$.}  \label{tab:table1}
	\end{center}
\end{table}

The finiteness of the Betti diagram can be deduced also from the general theory: in our situation it has as many columns as the projective dimension of $R$. In \cite{CM} the authors investigate to which extent the Betti diagram of $\mathbb{C}[S]$ reflects properties and invariants of $S$, particularly in the cases of both a symmetric semigroup and a complete intersection semigroup. Their main results in this direction are the following theorems.
\begin{theorem}
	\cite[Theorem 2.2]{CM} \(S\) is a symmetric semigroup if and only if the square is symmetric relative to its center.
\end{theorem}
\begin{theorem}\label{thm:squareCI}
	\cite[Theorem 3.1]{CM} Let \(S=\langle a_1,\dots,a_g\rangle\) be a complete intersection semigroup with \(I_S=(f_1,\dots,f_{g-1})\) and \(\deg(f_i)=m_i\) and assume \(m_1\leq m_2\leq\cdots\leq m_{g-1}\). Then, for \(-1\leq i\leq g-2\) one has 
	\[\beta_{i,m}=\begin{array}{c}
		\text{the number of ways in which \(m\) can be written as}\\
		m=m_{j_1}+\cdots+m_{j_l}\;\text{with}\;j_1<j_2<\cdots<j_l.
	\end{array}\]
\end{theorem}
From this perspective, it is reasonable to propose the following problems:
\begin{prob}\label{problem:CharCIcomplexes}
	Provide a characterization of a complete intersection semigroup (similar to Delorme's one) in terms of the simplicial complexes \(\Delta_m.\)
\end{prob}
In \cite{CM} there is a remark saying that the complete intersections are the only ones having the square of the form described in Theorem \ref{thm:squareCI}, this can in fact easily deduced from the definition of a complete intersection together with the analysis of its minimal free resolution. However, there is still no full characterization of complete intersections in terms of the simplicial complexes introduced there. In addition, it would be interesting to investigate the following naturally related questions.
\begin{question}
	Is there a description of Delorme's algorithm in terms of these simplicial complexes?
\end{question}
\begin{question}
	Is there any property of Dedekind and K\"{a}hler's differents that can be deduced from the Betti numbers?
\end{question}

Similar techniques to the ones developed by Campillo and Marijuan \cite{CM} were developed by Campillo and Giménez \cite{CG00} to describe the syzygies associated to an affine semigroup. It is then natural not only to explore the previous problem and questions for numerical semigroups but for the more general situation of an affine semigroup. Moreover, one could ask about the possibility to extend Herzog and Kunz's characterization to this more general setting. This would be quite a challenging problem, as in general affine semigroups do not have a conductor but they do have finite Betti diagram.


\section{Hilbert-Poincaré series}\label{sec:Poincare}
Let \(R=\bigoplus_{n\geq 0} R_n\) be a local Noetherian commutative ring graded by non-negative integers. Assume for simplicity that \(R_0=\mathbb{C},\) in general it is enough to assume that it is an arbitrary field. Each \(R_n\) is a finite--dimensional \(\mathbb{C}\)--vector space which naturally allows to define the Hilbert-Poincaré series as the formal power series
\[P_R(t)=\sum_{n=0}^{\infty}\dim_\mathbb{C}( R_n)t^n\in\mathbb{Z}[[t]].\]

Since \(R\) is finitely generated and graded, we can fix a system of homogeneous generators \(y_1,\dots,y_s\) such that \(y_i\notin R_0\) for \(i\geq 1.\) If one consider the ring of polynomials \(M=\mathbb{C}[x_1,\dots,x_s]\) such that \(\deg(x_i)=\deg(y_i)=a_i,\) we have the natural graded structure on \(M\) and a canonical degree preserving epimorphism \(M\rightarrow R.\) As a consequence, the commonly known as the Hilbert-Serre theorem (see e.g \cite[Theorem 8.20]{SMbook}), which is in fact deduced from the Hilbert Syzygy theorem, states that we can write 
\[P_R(t)=\frac{Q(t)}{\prod_{i=1}^{s}(1-t^{a_i})}\]
with \(Q(t)\in\mathbb{Z}[t]\) being a polynomial. Besides the additivity property under exact sequences, the following property is also one of the most useful properties of Hilbert-Poincaré series
\begin{theorem}\cite[Theorem 3.1]{St78hilbert}\label{thm:Poincaremultiplication}
	Let \(R\) be a graded \(\mathbb{C}\)--algebra and let \(f\in R\) be a homogeneous element of degree \(m>0.\) Then,
	\[P_{R}(t)=\frac{P_{R/f}(t)-t^mP_{\operatorname{Ann}(f)}(t)}{1-t^m}\]
\end{theorem}
\begin{proof}
	We have the following exact sequence:
	\[0\rightarrow \operatorname{Ann}(f)\rightarrow R\xrightarrow{\cdot f} R\rightarrow\frac{R}{(f)}\rightarrow0.\]
	The multiplication by \(f,\) can be made homogeneous of degree \(0\) just by shifting the grading in \(R\) by the degree of \(f,\) i.e. we consider \(\widetilde{R}=R(-m).\) Analogously, the ideal \(\operatorname{Ann}(f)\) can be also considered as an ideal of \(\widetilde{R},\) let us denote it by \(\widetilde{\operatorname{Ann}(f)}.\)  Therefore,
	\[P_{\widetilde{\operatorname{Ann}(f)}}(t)=t^mP_{\operatorname{Ann}(f)}(t)\quad\text{and}\quad P_{\widetilde{R}}(t)=t^mP_R(t).\]
	Applying the additivity of \(\dim_\mathbb{C}\) to the exact sequence 
	\[0\rightarrow \widetilde{\operatorname{Ann}(f)}\rightarrow \widetilde{R}\xrightarrow{\cdot f} R\rightarrow\frac{R}{(f)}\rightarrow0,\]
	we obtain 
	\[P_R(t)+P_{\widetilde{\operatorname{Ann}(f)}}(t)=P_{\widetilde{R}}(t)+P_{R/(f)}(t).\]
	From which the claimed equality is now deduced.
\end{proof}
Theorem \ref{thm:Poincaremultiplication} has as an immediate consequence that we can easily compute the Hilbert-Poincaré series of any complete intersection \(\mathbb{C}\)--algebra generated by an homogeneous ideal.

\begin{corollary}
	Let \(M=\mathbb{C}[x_1,\dots,x_s]\) with \(\deg x_i=a_i\) and let \(f_1,\dots,f_r\) be a homogeneous regular sequence with \(\deg f_i=m_i.\) Let \(R=M/(f_1,\dots,f_r)\) be the complete intersection with the induced grading. Then,
	\[P_R(t)=\frac{\prod_{i=1}^{r}(1-t^{m_i})}{\prod_{i=1}^{s}(1-t^{a_i})}.\] 
\end{corollary}

In 1978, Stanley \cite[pg.64]{St78hilbert} posed the following interesting question: 
\begin{question}\label{ques:santaley}(\textbf{Stanley} \cite[pg.64]{St78hilbert})
	If a graded algebra R has the Hilbert function of a complete intersection, under what circumstances can we conclude that R actually is a complete intersection ? 
\end{question}
\noindent The following examples show some difficulties to answer Question \ref{ques:santaley}:
\begin{itemize}
	\item \cite[Example 3.7]{St78hilbert} If \(R_1=k[x,y]/(x^2,y^2)\) and \(R_2=k[x,y]/(x^3,xy,y^2)\) each with the standard grading \(\deg x=\deg y=1,\) then 
	\[P_{R_1}(t)=P_{R_2}(t)=(1+t)^2.\]
	However, \(R_1\) is a complete intersection but \(R_2\) is not even Gorenstein.
	\item   \cite[Example 3.9]{St78hilbert} Let \(R=k[x_1,x_2,x_3,x_4,x_5,x_6,x_7]/I\) where again \(R\) is considered with the standard grading and \[I=(x_1x_5-x_2x_4,\ x_1x_6-x_3x_4,\ x_2x_6-x_3x_5,\ x_1^2x_4-x_5x_6x_7,\ x_1^3-x_3x_5x_7).\] Then \(R\) is Gorenstein but not a complete intersection, however it has the Hilbert-Poincaré series of a complete intersection 
	\[P_R(t)=\frac{(1+t)^3}{(1-t)^4}.\]
\end{itemize}
\begin{rem}
	Recently, Borz\`{i} and D'Al\`{i} \cite{BA21} have answered Question \ref{ques:santaley} in the case of Koszul algebras, i.e. they showed that a Koszul algebra is a complete intersection if and only if the numerator \(Q(t)\) of the Hilbert-Poincaré series has all of its roots on the unit circle.
\end{rem}
In contrast to the complete intersection case, one can actually characterize Gorenstein rings from certain symmetry property in its Hilbert-Poincaré series.
\begin{theorem}\label{thm:hilbersym}
	\cite[theorem 4.4]{St78hilbert} Let \(R\) be a graded algebra. Suppose \(R\) is Cohen-Macaulay of Krull dimension \(d.\) Then, \(R\) is Gorenstein if and only if for some \(a\in \mathbb{Z},\)
	\[P_R\left(\frac{1}{t}\right)=(-1)^dt^aP_R(t).\]
\end{theorem}
\subsection{Hilbert-Poincaré series of numerical semigroups and Alexander polynomials}\label{subsec:poincaresemigroup}
Let us return to our case of study, i.e. where \(R=\mathbb{C}[S]\) is the semigroup algebra of a numerical semigroup \(S.\) In this case, it is well known that 
\[\dim_\mathbb{C}R_n=\left\{\begin{array}{cc}
	1&\text{if}\quad n\in S\\
	0&\text{if}\quad n\notin S.
\end{array}
\right.\] 
Therefore, in this particular case, the Hilbert-Poincaré series coincides with the generating series of the semigroup
\[P_R(t)=\sum_{s\in S}t^s.\]

In \cite{RosGluing}, Rosales introduced an operation, called gluing, among semigroups in \(\mathbb{N}^n\). \footnote{ Notably, Delorme \cite{delormeglue} and Herzog \cite{herzog} had previously applied this operation in their discussions on the set of relations within a semigroup algebra but not under this name (see Section \ref{sec:CI}).} We now recall the definition of gluing and its utility in the context of computing Hilbert-Poincaré series under this operation. Let \(S,S_1,S_2\subset\mathbb{N}^n\) be three finitely generated submonoids of \(\mathbb{N}^n.\) Let \(I_S,I_{S_1},I_{S_2}\) be the defining ideals of their corresponding semigroup algebras. One says that \(S\) is the gluing of \(S_1\) and \(S_2\) if and only if \(I_S=I_{S_1}+I_{S_2}+\rho\) where \(\rho=\mathbf{x}^{\mathbf{v}}-\mathbf{y}^{\mathbf{w}} \) is a binomial in separated variables, i.e \(\mathbf{x}^{\mathbf{v}}\in\mathbb{C}[S_1]\) and \(\mathbf{y}^{\mathbf{w}}\in\mathbb{C}[S_2].\) In short, \(\mathbb{C}[S]\) is the gluing of \(\mathbb{C}[S_1]\) and \(\mathbb{C}[S_2]\) if and only if 
\[\mathbb{C}[S]=\frac{\mathbb{C}[S_1]\otimes \mathbb{C}[S_2]}{\rho}.\]
As the gluing condition is quite strong from the algebraic point of view, not every semigroup in \(\mathbb{N}^n\) can be glued. This is mainly due to the restrictions prescribed by Auslander-Buchsbaum formula as one can see in \cite{Philippe2}. However, in the case where \(n=1,\) i.e. for numerical semigroups, this operation can be always performed as Auslander-Buchsbaum formula does not provide any restriction in this particular case. 

For numerical semigroups (and in fact for affine ones) obtained by gluing one can easily compute its Hilbert-Poincaré series as follows.
\begin{proposition}
	Let \(S\) be a numerical semigroup generated by \(A=\{a_1,\dots,a_g\}\subset\mathbb{N}.\) Assume that there exists a partition \(C=A_1\sqcup A_2\) with \(A_1\neq \emptyset\neq A_2\) such that \(S\) is the gluing of \(S_{A_1}=\langle A_1\rangle/(\gcd(A_1))\) and \(S_{A_2}=\langle A_2\rangle/\gcd(A_2).\) Let us denote by \(d_{A_i}=\gcd(A_i)\) and by \(d=\deg(\rho)\) where \(\rho\) is the binomial such that \(I_S=I_{S_{A_1}}+I_{S_{A_2}}+\rho.\) Then, 
	\[P_S(t)=(1-t^d)P_{S_{A_1}}(t^{d_1})P_{S_{A_2}}(t^{d_2})\]
\end{proposition}
\begin{proof}
	The gluing operation gives rise to the following exact sequence:
	\[0\rightarrow \mathbb{C}[S_1]\otimes \mathbb{C}[S_2]\xrightarrow{\cdot\rho}\mathbb{C}[S_1]\otimes \mathbb{C}[S_2]\rightarrow \mathbb{C}[S]=\frac{\mathbb{C}[S_1]\otimes \mathbb{C}[S_2]}{\rho}\rightarrow 0.\]
	Arguing as in the proof of Theorem \ref{thm:Poincaremultiplication}, we obtain the required equality.
\end{proof}

Following Delorme's algorithm (see Section \ref{subsec:delorme}) and the previous discussion, it is trivial to see that one can iteratively construct the Hilbert-Poincaré series of a complete intersection numerical semigroup following the gluing construction. More concretely, if \(S=\langle a_1,\dots,a_g\rangle\) is a complete intersection numerical semigroup and \(m_1,\dots,m_{g-1}\) are the degrees of the monomials corresponding to a distinguished sequence, then 

\[P_S(t)=\frac{\prod_{i=1}^{g-1}(1-t^{m_i})}{\prod_{i=1}^{s}(1-t^{a_i})}.\]

In any case, one can observe that these kind of examples are easy consequences of the more general Theorem \ref{thm:Poincaremultiplication}. Quite recently, several authors \cite{cyclo1,cyclo2,cyclo3} have explored Stanley's Question \ref{ques:santaley} in the particular case of the graded algebra associated to a numerical semigroup. They conjecture that \(Q(t)\) has as only roots of unity if and only if \(S\) is a complete intersection numerical semigroup. However, this problem is still quite an open problem.

\subsubsection{\textbf{Alexander polynomials}}\label{subsubsec:alexanderpoly}

The Alexander polynomial \(\Delta_K(t)\) is an algebraic invariant of a knot \(K\) in the sphere $\mathbb{S}^3$ defined by Alexander in 1928 \cite{alexoriginal}. The definition of the Alexander polynomial bases on the notion of universal abelian covering \(\eta:\widetilde{X}\rightarrow X\) of the knot complement \(X=\mathbb{S}^3\setminus K.\) The group of covering transformations \(H_1(X;\mathbb{Z})=\mathbb{Z}\) is a free abelian multiplicative group on the symbol \(\{t\}\) where \(t\) is geometrically associated with an oriented meridian of the knot. In this way, if \(\widetilde{p}\) is a typical fiber of \(\eta\) then the group \(H_1(\widetilde{X},\widetilde{p};\mathbb{Z})\) becomes a module over \(\mathbb{Z}[t^{\pm 1}].\) As showed in \cite[Chap. I, Prop. 5.1 (see also the proof of that proposition)]{EN} we have the \(\mathbb{Z}[t^{\pm 1}]\)--module isomorphism   \(H_1(\widetilde{X},\widetilde{\rho};\mathbb{Z}))\simeq H_1(\widetilde{X};\mathbb{Z})\oplus \mathbb{Z}[t^{\pm 1}]\) which implies \(F_1(H_1(\widetilde{X},\widetilde{\rho};\mathbb{Z}))\simeq F_0(H_1(\widetilde{X};\mathbb{Z})).\) The Alexander polynomial \(\Delta_K(t)\) is then defined as the greatest common divisor of the Fitting ideal \(F_0(H_1(\widetilde{X};\mathbb{Z})).\) In particular, observe that  \(\Delta_K(t)\) is then well-defined up to multiplication by a unit of \(\mathbb{Z}[t^{\pm 1}].\) For generalities about knots we refer to \cite{EN,BurdeKnots,Murasugibook}.
\medskip

The relation between the Poincaré series associated to some geometric contexts and the Alexander polynomials has been profusely studied by Campillo, Delgado and Gusein-Zade later on, see \cite{CDG99a, CDG99b, CDG00, CDG02, CDG03a, CDG04, CDG05, CDG07}. Surprisingly, in one of his investigations they showed that the Poincar\'e series ---which turns out to be a polynomial in the case of a plane curve singularity with more than one branch--- coincides with the Alexander polynomial associated to the link of the singularity \cite{CDGduke}. If \(C\) is an irreducible plane curve singularity, \(C\cap \mathbb{S}^3\) naturally defines a knot \(K\) which is called the knot associated to \(C.\) In this case, Campillo, Delgado and Gusein-Zade theorem states that 
\[(1-t)P_C(t)=(1-t)\left(\sum_{s\in S_C}t^s\right)=\Delta_K(t).\]

The knots arising as the associated knot of a plane curve singularity are called algebraic knots \cite{EN}. Thus, it is natural to ask up to what extent one can naturally assign to any knot (not only an algebraic one) a certain semigroup for which its Poincaré series is related to the Alexander polynomial of the knot. In this context, Wang \cite{Wang18}  has proposed the study of this situation in the case of \(L\)--space knots. The notion of \(L\)--space was first defined in \cite{OStopology} by  Ozsváth and Szabó. An \(L\)--space is a rational homology sphere having the Heegaard Floer homology of a lens space. A knot \(K\) in \(\mathbb{S}^3\) is an \(L\)--space knot if it admits an \(L\)--space surgery. It is not our purpose to go into the details of this definition, therefore we refer the reader to \cite{OStopology,OSadvances} for further details.
\medskip

For an \(L\)--space knot \(K\), Ozsváth and Szabó \cite[Theorem 1.2]{OStopology} showed that the Alexander polynomial can be expressed in the form \(\Delta_K(t)=\sum_{i=0}^{2n-1}(-1)^{i}t^{\alpha_i}.\)  Once this expression of the Alexander polynomial is fixed, Wang \cite{Wang18} defined the formal semigroup of an \(L\)--space knot \(K\) as the set \(S_K\subset \mathbb{Z}_{\geq 0}\) satisfying

\[\frac{\Delta_K(t)}{1-t}=\sum_{s\in S_K}t^s.\]

Naturally, Wang asks whether \(S_K\) is actually a semigroup or not, i.e. closed by addition. A first observation is that not every formal semigroup is a semigroup. For example, the Pretzel knot \(P(-2,3,7)\) \cite[Example 2.3]{Wang18} has as formal semigroup the set \(S_P=\{0,3,5,7,8,10\}\cup\mathbb{Z}_{>10}.\) Clearly, \(S_P\) is not a semigroup as \(6\notin S_P\). In fact, Wang points out that for any odd integer \(n\geq 7,\) the formal semigroup of a pretzel knot \(P(-2,3,n)\) is not a semigroup. Several families of knots having its formal semigroup a true semigroup are given in \cite{Wang18}; all of them being iterated torus knots. This motivated Wang \cite[Question 2.8]{Wang18} to ask for the existence of an \(L\)--space knot \(K\) with \(S_K\) being a semigroup but \(K\) not being an iterated torus knot. This question has been answered by Teragaito \cite{Teragaito22}, who has provided an infinite family of hyperbolic \(L\)--space knots (in particular not iterated torus knots) whose formal semigroups are semigroups. Let us analyze Teragaito's examples.
\medskip

First, Teragaito showed that the formal semigroup of the hyperbolic L-space knots \(K_n\) with \(n\geq 1\) in \cite{censusknots} are all semigroups \cite[Theorem 4.1]{censusknots}. These semigroups are generated as \(S_{K_n}=\langle 4, 4n+2,4n +5\rangle.\) Observe that, thanks to the discussion of Section \ref{sec:CI} and Herzog's characterization of complete intersection semigroup with embedding dimension \(3\) \cite[sec. III]{herzog}, as \(\gcd(4,4n+2)=2\) for any \(n\geq 1\) and \(2n+5\in \langle 2, 2n+1 \rangle\) then  \(S_{K_n}\) is a complete intersection numerical semigroup for any \(n\geq 1.\) After these examples, Teragaito in \cite{Teragaito22} used this family to provide a new family of hyperbolic \(L\)--space knots whose formal semigroups are semigroups with embedding dimension equal to \(5,\) this family of semigroups is of the form \[S_n=\langle 6, 6n+4,6n+8,12n+11,12n+15\rangle\quad \text{for}\quad n\geq 1.\] Obviously, the semigroups \(S_n\) are symmetric for any \(n\geq 1\) as we will see in Proposition \ref{prop:necessaryformalsemi}. Let us show that they are not complete intersection by using Delorme's algorithm explained in section \ref{subsec:delorme}. With the notation of section \ref{subsec:delorme}, let us fix \(n\geq 1\) and let \(P_5\) be the partition of the set of generators into sets of cardinal one. We compute now the syzygy candidates:
\[
\begin{array}{c}
	m_{6}=12n+12, \quad m_{6n+4}=12n+8,\quad m_{6n+8}=12n+16,\\
	m_{12n+11}=24+22, \quad \text{and} \quad m_{12n+15}=24n+30.
	
\end{array}
\]
As for any \(L,L'\in P_5\) we have \(m_L\neq m_{L'},\) Delorme's algorithm shows that \(S_n\) is not a complete intersection numerical semigroup for any \(n\geq 1.\)
\medskip

From our point of view, the classification of semigroups which are formal semigroups of \(L\)--space knots could be an interesting topic both in the theory of numerical semigroups and knot theory. Seifert showed that the Alexander polynomial of a knot is symmetric \cite{Seifert34}, this result was extended by Torres \cite{Torres53} (see also \cite{TorresFox54}) to the case of links with several components. Therefore, a first necessary condition for a semigroup to be realizable is to be symmetric, i.e. the semigroup algebra must be Gorenstein as we have previously explained.

\begin{proposition}\label{prop:necessaryformalsemi}
	Let \(K\) be a \(L\)--space knot and assume that its formal semigroup \(S_K\) is a semigroup. Then, \(S_K\) is symmetric.
\end{proposition}
\begin{proof}
	By Theorem \ref{thm:hilbersym}, it is enough to show that \(P_{S_K}(1/t)=-t^{a}P_{S_K}(t)\) for some \(a\in \mathbb{Z}.\) By \cite[Theorem 1.2]{OStopology} the symmetrized Alexander polynomial is of the form \(\Delta_K(t)=\sum_{i=1}^{s}(-1)^k(t^{n_k}+t^{-n_k})\) for some increasing sequence of positive integers \(0<n_1<\cdots<n_s.\) As \(\Delta_K(t)\) is defined up to units in \(\mathbb{Z}[t^{\pm 1}],\) thus \(\Delta_K(t)\sim\Delta_K(t)\cdot t^{n_s}=\sum_{k=0}^{2s-1}(-1)^kt^{\alpha_k}\) with \(\alpha_{s-k}=n_s-n_k\) for \(k=1,\dots,s\) and \(\alpha_{s+k-1}=n_s+n_k\) for \(k=1,\dots,s\). If \(S_K\) is a semigroup then by definition \(P_{S_K}(t)=\Delta_K(t)t^{n_s}/(1-t),\) as the definition is given for a fixed expression of the Alexander polynomial as polynomial in \(\mathbb{Z}[t].\) Thus, if we choose \(a=1-2n_s\) we have
	\[-t^{1-2n_s}P_{S_K}(t)=-\frac{\Delta_K(t)t^{1-n_s}}{1-t}=-\frac{\Delta_K(1/t)t^{-n_s}}{(1/t)-1}=P_{S_K}(1/t)\]
\end{proof}

Recall that in the case of embedding dimension \(e(S)= 3\), Herzog's Theorem \cite[sec. III]{herzog} also states that a numerical semigroup with \(e(S)= 3\) is a complete intersection if and only if it is symmetric. Therefore, as symmetry is a necessary condition for a semigroup to be the formal semigroup of a knot, it is obvious that in the case of embedding dimension \(3,\) the semigroups appearing as formal semigroup must be complete intersections. Moreover, Teragaito's example leads us to propose the following conjecture:
\begin{conj}
	Let \(K\) be a hyperbolic \(L\)--space knot. Assume that its formal semigroup \(S_K\) is a semigroup with embedding dimension bigger or equal than \(4.\)  Then, \(S_K\) is not a complete intersection semigroup.
\end{conj}

While the Alexander polynomial is a reasonable and interesting connection between knot theory and numerical semigroup theory, we propose the following alternative approach relating knots and numerical semigroups. To do that, we need first to introduce the following definition. Given a knot \(K\), the abelianization of its knot group is always isomorphic to \(\mathbb{Z},\) and as already mentioned we can think in \(\mathbb{Z}\) as the multiplicative abelian group on the symbol \(\{t\}.\) Let \(G=\pi_1(X)\) be the fundamental group of the knot complement. Assume we represent \(\mathbb{Z}=<t>\) as a multiplicative group. We will say that a numerical semigroup \(S=\langle a_1,\dots, a_g\rangle\) is topologically realizable if there exists a knot \(K\) such that the following conditions holds:
\begin{enumerate}
	
	\item There exists \(G=\pi_1(\mathbb{S}^3\setminus K)=\langle \alpha_1,\dots,\alpha_s,\;| \ r_i(\alpha_1,\dots,\alpha_s)\;\text{for}\; i=1,\dots,l\rangle\) a presentation of the group \(G\) with \(s\) generators and \(l\) binomial relations in separated generators, i.e.
	\begin{equation*}
		\begin{split}
			r_i(\alpha_1,\dots,\alpha_s):&\quad p_{i,1}(\alpha_{i_{1},\dots,i_{k}})=p_{i,2}(\alpha_{j_{1},\dots,j_{t}}) \\ &\text{with}\ \{\alpha_{i_{1},\dots,i_{k}}\}\cap \{\alpha_{j_{1},\dots,j_{t}}\}=\emptyset, \ t+k\leq g+1 
		\end{split}
	\end{equation*}
	and such that \(s=g\) is the number of minimal generators of \(I_S,\) the defining ideal of \(\mathbb{C}[S].\)
	\item The abelianization map is such that \(\alpha_i\mapsto t^{a_i}.\)
	\item If we consider the ring of polynomials \(\mathbb{C}[x_1,\dots,x_s]\) with the degrees induced by the abelianization map, then the relations of the group generate the defining ideal of \(S,\) i.e. 
	\[I_S=(p_{1,1}(x_1,\dots,x_s)-p_{1,2}(x_1,\dots,x_s),\dots,p_{l,1}(x_1,\dots,x_s)-p_{l,2}(x_1,\dots,x_s)).\]
\end{enumerate}

In brief, a numerical semigroup will be said to be topologically realizable if its semigroup algebra can be computed through the abelianization map of the fundamental group of the complement of a certain knot.
\medskip

In \cite[Section 5.2]{AMknots}, the author and Moyano-Fernández have recently showed that the semigroup \(S=\langle\obeta_0,\dots,\obeta_g \rangle\) of an irreducible plane curve singularity is always realizable. To do so, we provided the following topological interpretation of the gluing construction for the semigroup of an irreducible plane curve singularity. Following \cite[section 5.2]{AMknots}, let \(C\) be an irreducible plane curve singularity and let \(G:=\pi_1(X)\) be the fundamental group of the knot complement $X$ of its associated knot. Zariski \cite[\S 4]{Zar32} proved that
\[G=\langle b_1,u_1,\dots,u_g\;|\; u_i^{n_i}=b_i^{q_i}u_{i-1}^{n_{i-1}n_i}\; \ \text{for}\ \;i=1,\dots,g\rangle,\]
where \(u_0=1,\) the elements \(b_2,\dots,b_g\) are determined by the relations
\[b_{i+1}b_i^{y_i}u_{i-1}^{n_{i-1}x_i}=u_i^{x_i}\quad\text{for}\;i=1,\dots,g-1,\]
in which the positive integers \(x_i,y_i\) are defined from a Puiseux parameterization of the curve. Let us denote it by \(\langle t\rangle=\operatorname{Ab}(G).\) Hence, in \(\operatorname{Ab}(G)\) all the elements become powers of \(t\) and one can check that \(b_i=t^{n_i\cdots n_g}\) and \( u_i=t^{\obeta_i}.\)

Then, in \cite{AMknots} we prove that \(I_S\) is generated by the relations defining the fundamental group \(G\) i.e. we have \(I_S=(u_i^{n_i}=b_i^{q_i}u_{i-1}^{n_{i-1}n_i})_{i=1}^{g}.\) In this way, we showed that the gluing construction comes from a satellization operation in the iterated torus knots via the application of the Seifert-Van Kampen theorem. This shows that the amalgamated decomposition of the fundamental group modulo a relation translates into the tensor product decomposition modulo the relation of the semigroup algebra, which is the algebraic interpretation of the gluing construction on the semigroup structure. From this perspective, we believe the following problem could be of interest. 
\begin{prob}
	To characterize those numerical semigroups that can be topologically realizable.
\end{prob}


\section{Deformations of monomial curves}\label{sec:deformations}

Let \(C^{S}: (t^{a_1},\dots,t^{a_g})\subset\mathbb{C}^g\) be a monomial curve, let \(R=\mathbb{C}[S]\) be its semigroup algebra and, as in section \ref{sec:monomial}, let us denote by \(M(S)=\mathbb{C}[x_1,\dots,x_g]\) the ring of polynomials in \(g\) variables with the grading induced by \(S.\) As we mentioned in section \ref{subsec:deformations}, the Zariski tangent space \(T^1_R\) to the base space of the miniversal deformation of \(C^S\) is defined by dualizing the exact sequence
\begin{equation}
	\label{eqn:exactT1}
	I_S/I^2_S\rightarrow\Omega^1_{M(S)}\otimes_{\raisebox{0.5pt}{\scriptsize $M(S)$}}R\rightarrow\Omega^1_{R}\rightarrow 0.
\end{equation}

Following Pinkham \cite[Chapter 13]{Pinkham1}, one can construct a compactification of a moduli space for smooth algebraic curves with a Weierstrass point of fixed semigroup \(S\) which is related to the negatively graded part of \(T_R^1.\) To do that, let us first recall the definition of Weiertrass semigroup. Let \(p\) be a point on a smooth projective curve \(C\). The Weiertrass semigroup of the point \(p\) is defined as
\[S_p:=\left\{h\in\mathbb{N}\;\colon \;\begin{array}{c}
	\text{there exists a meromorphic function defined on \(C\),}\\
	\text{ holomorphic on \(C\setminus p\) with a pole of degree \(h\) at \(p\)}
\end{array}\right\}.\]

The aim is to construct a moduli space related to this semigroup of poles. To do so, first we need also to recall the construction of the versal deformation. Following Pinkham \cite[Chapter 2]{Pinkham1}, consider a homogeneous basis \(s_1\dots,s_\tau\in\mathbb{C}[x_1,\dots,x_g]^{g-1}\) of \(T^1_R\), where we put \(s_i=(s^1_i,\dots,s^{g-1}_i)\) for $i=1,\ldots , \tau$. Let us denote by \((f_{0,1},\dots,f_{0,l_1})=I_S\) a minimal set of generators for the defining ideal of the curve \(C^S.\)  Then, the versal deformation of \(C^S\) can be described as follows
\begin{equation}\label{eqn:eqdeformation}
	\begin{array}{cc}
		F_1(\mathbf{x},\mathbf{w})=&f_{0,1}(\mathbf{x})+\sum_{j=1}^{\tau}w_js^{1}_j(\mathbf{x}),\\
		\vdots&\vdots\\
		F_{l_1}(\mathbf{x},\mathbf{w})=&f_{0,l_1}(\mathbf{x})+\sum_{j=1}^{\tau}w_js^{k}_j(\mathbf{x}).
	\end{array}
\end{equation}
Let \((\mathcal{X},\mathbf{0}):=V(F_1,\dots,F_{l_1})\subset(\mathbb{C}^{g}\times\mathbb{C}^{\tau},\mathbf{0})\) be the zero set of \(F_1,\dots,F_{l_1}\), then the deformation defined by  \((C^{S},\mathbf{0})\xrightarrow{i}(\mathcal{X},\mathbf{0})\xrightarrow{\phi}(\mathbb{C}^\tau,\mathbf{0})\) is the versal deformation of \((C^S,\mathbf{0})\), where \(i\) is induced by the inclusion and \(\phi\) by the natural projection. 

In fact, as we have chosen an homogeneous basis of \(T^1_R,\) we can also choose weights for the parameter space. If one chooses \(\deg(w_j)=-\deg(s_j)\), then we endow the algebra \(\mathbb{C}[x_1,\dots,x_g,w_1,\dots,w_\tau]\) with the unique grading for which \(\deg(x_i)=a_i\) and the \(F_i\) are homogeneous with \(\deg(F_i)=\deg(f_{0,i})\). Under this grading, we obtain a partition of the base space \(\mathbb{C}^\tau\) into two parts. As showed by Pinkham \cite[Lemma 12.5]{Pinkham1} \(T^1_R(0)=0\) for any monomial curve, thus we can split the parameter space into the sets

\begin{align*}
	P_{+}:=&\{j\in\{1,\dots,\tau\}:\;\deg(w_j)<0\}\\
	P_{-}:=&\{j\in\{1,\dots,\tau\}:\;\deg(w_j)> 0\}.
\end{align*}

Denote by $\tau_{+}(S):=\tau_{+}:=|P_{+}|$ and $\tau_{-}(S):=\tau_{-}:=|P_{-}|$. There is a natural action of the group \(\mathbb{C}^{\ast}\) over \((\mathcal{X},\mathbf{0})\) which is compatible with the previous construction and that induces the natural action on the central fiber \(\phi^{-1}(0)\cong C^{S}\). Following Pinkham \cite{Pinkham1}, the main goal is now to construct a certain moduli space from the negatively graded part of the deformation. We refer to \cite[Chapter 13]{Pinkham1} and the references therein for further considerations. 
\medskip

Let us denote by $\mathcal{M}_{g,1}$ the coarse moduli space of smooth projective curves $C$ of genus $g$ with a section i.e. of pointed compact Riemann surfaces of genus $g$. The aim is to relate a certain subscheme of  $\mathcal{M}_{g,1}$ with the negatively graded part of the deformation. To this purpose we need first to consider the base change in the deformation induced by the inclusion map defined as \(V_-:=(\mathbb{C}^{\tau_-}\times\{\mathbf{0}\},0)\hookrightarrow (\mathbb{C}^{\tau},\mathbf{0})\), on account of the diagram
\[
\xymatrix{
	(C^{S},\mathbf{0})\ar[dr]\ar[r]&(\mathcal{X},\mathbf{0})\ar[rr]& &(\mathbb{C}^{\tau},\mathbf{0})  \\
	& (\mathcal{X}_S,\mathbf{0}):=(\mathcal{X},\mathbf{0})\times_{(\mathbb{C}^{\tau},\mathbf{0})}(V_-,\mathbf{0})\ar[rr]\ar[u]& &(V_-,\mathbf{0}).\ar[u]}
\]
Let us denote by \(G_S: \mathcal{X}_S\rightarrow V_-\) the deformation induced by this base change. Observe that this deformation can be described in terms of the equations by making \(w_j=0\) for all \(j\in P_+.\) This is now a negatively graded deformation. At this point, we must projectivize the fibers of \(G_S\) without projectivizing the base space \(V_-.\) This can be done by replacing \(s_j\) with \(s_j(x_1,\dots,x_g)X_{g+1}^{-\deg s_j}.\) Observe that we have the inclusion \(\overline{\mathcal{X}}_S\subset \mathbb{P}^{g+1}\times V_-\) where the ring \(\mathbb{C}[x_1,\dots,x_g,X_{g+1}]\) has \(\deg x_i=a_i\) and \(\deg X_{g+1}=1.\) According to Pinkham \cite[Proposition 13.4, Remark 10.6]{Pinkham1} the morphism
$$
\pi: \overline{\mathcal{X}}_S\longrightarrow V_{-}
$$
is flat and proper, its fibres are reduced projective curves, and all those fibres which are above a given $\mathbb{G}_m$-orbit of $V_{-}$ are isomorphic. This leads Pinkham \cite[Theorem 13.9]{Pinkham1} to prove the following announced correspondence between negatively graded deformations and moduli spaces.

\begin{theorem}[Pinkham]\label{thm:Pinkham}
	Let $\mathcal{M}_{g,1}$ be the coarse moduli space of smooth projective curves $C$ of genus $g$ with a section i.e. of pointed compact Riemann surfaces of genus $g$. Let $S$ be a numerical semigroup and set the subscheme of $\mathcal{M}_{g,1}$ parameterizing pairs
	$$
	W_{S} = \Big  \{ (X_0,p) : X_0 \ \mbox{is a smooth projective curve of genus} \ g, \ \mbox{and} \ p\in X_0 \ \mbox{with} \  S_p=S  \Big \},
	$$
	where $S_p$ is the Weierstrass semigroup at the point $p$. Moreover, write $V^{s}_{-}$ for the open subset of $V_{-}$ given by the points $u\in V_{-}$ such that the fibre of $\overline{\mathcal{X}}_S \to V_{-}$ above $u$ is smooth. This is $\mathbb{G}_m$ equivariant, and so there exists a bijection between $W_S$ and the orbit space $V^{S}_{-} / \mathbb{G}_m$.
\end{theorem}

After Pinkham's results, the moduli space \(W_S\) and its dimension has been widely studied in the literature. We refer to St\"ohr \cite{stohr1} and Rim and Vitulli \cite{RV77} and the references therein, as well as Contiero and St\"ohr \cite{sthor2} and Stevens \cite{stevens}, to obtain a good general view of the state of the art of this problem. In any case, a full understanding of \(W_S\) and Theorem \ref{thm:Pinkham} is still being the subject of a very active research area.

\mysubsection{Moduli of irreducible plane curves}
Let \(S\) be the semigroup of an irreducible plane curve singularity, recall that it is a semigroup satisfying the conditions \eqref{cond:beta}. In an analogous way to Pinkham, Teissier \cite{teissierappen} used the positively graded part of the deformation to produce a representation of the analytic moduli associated to an irreducible plane curve with fixed semigroup. To this purpose, instead of \(V_-,\) he considered \(V_+:=(\mathbb{C}^{\tau_+}\times\{\mathbf{0}\},0)\hookrightarrow (\mathbb{C}^{\tau},\mathbf{0})\) and the induced base change
\[
\xymatrix{
	(C^{S},\mathbf{0})\ar[dr]\ar[r]&(\mathcal{X},\mathbf{0})\ar[rr]& &(\mathbb{C}^{\tau},\mathbf{0})  \\
	& (\mathcal{X}_{S,+},\mathbf{0}):=(\mathcal{X},\mathbf{0})\times_{(\mathbb{C}^{\tau},\mathbf{0})}(V_+,\mathbf{0})\ar[rr]\ar[u]& &(V_+,\mathbf{0}).\ar[u]}
\]

In this case, if we denote by \(G_{S,+}: (\mathcal{X}_{S,+},\mathbf{0})\rightarrow(V_{+},\mathbf{0})\) the induced deformation, it is easy to see that the fibers has constant semigroup \(S.\) This is the reason why this deformation is called \emph{miniversal constant semigroup deformation} of \((C^S,\mathbf{0}).\) To analyze the different analytic types of curves with fixed semigroup \(S\) it is enough to study constant semigroup deformations of \(C^S\). This is possible thanks to the following theorem.

\begin{theorem}[{\cite[I.1]{teissierappen}}]\label{thm:Teissiergenericdeformations}
	Every irreducible plane curve singularity \( (C, \boldsymbol{0}) \) with semigroup \( S \) is isomorphic to the generic fiber of a one parameter complex analytic deformation of \( (C^S, \boldsymbol{0}) \). 
\end{theorem}

\begin{rem}
	As remarked by Teissier \cite{teissierappen}, the above statement is a short-hand way of stating the following: for every branch \( (C, \boldsymbol{0}) \) with semigroup \( S \) there exists a deformation \(\phi: (X, \boldsymbol{0}) \longrightarrow (D, \boldsymbol{0})\) of \(\mathbb{C}^S\), with a section \(\sigma\), such that for any sufficiently small representative \(\widetilde{\phi}\) of the germ of \(\phi,\) \((\widetilde{\phi}^{-1}(v),\sigma(v))\) is analytically isomorphic to \( (C, \boldsymbol{0}) \) for all \(v\neq 0\) in the image of \(\widetilde{\phi}\).
\end{rem}

Following Teissier \cite[Chap. II, Sec. 2]{teissierappen}, analytic equivalence of germs induces an equivalence relation \(\sim\) on \(V_{+}\) as follows: \(w\sim w'\) if and only if  the germs \((G_{S,+}^{-1}(w),\sigma_{S,+}(w))\) and \((G_{S,+}^{-1}(w'),\sigma_{S,+}(w'))\) are analytically isomorphic. Thus, Teissier call \(\widetilde{M}_{S}:=V_{+}/\sim\) the moduli space associated to the semigroup \(S\). Let \( \mathrm{m} : V_{+} \longrightarrow \widetilde{M}_S \) be the natural projection and let \( D^{(2)}_S \) be the following subset of \( V_{+} \)

\[ D^{(2)}_S := \{ \boldsymbol{v} \in V_{+} \ |\ (G^{-1}(\boldsymbol{v}), \boldsymbol{0})\ \textrm{is a plane branch} \}.  \]
Then, Teissier proves in \cite[Chap. II, 2.3 (2)]{teissierappen} that \( D^{(2)}_S \) is an analytic open dense subset of \( D_S \) and that \( \mathrm{m}(D^{(2)}_S) \) is the moduli space \( M_S \) of plane branches with semigroup \( S \) in the sense of Ebey \cite{ebey} and Zariski \cite{Zarbook}. 
\medskip

While different in nature, in \cite{AMminiversal} we showed the close connection between the moduli space defined by Pinkham and the one defined by Teissier if one consider plane curves with only one place at infinity. It would be an interesting problem to consider this kind of relations in a more general setting.

\subsection{The miniversal deformation of a complete intersection singularity}

Let us now assume \(C^S\) to be a complete intersection monomial curve. In that case, as the defining ideal \mbox{\(I_S=(f_1,\dots,f_{g-1})\)} defines a regular sequence, and the exact sequence \eqref{eqn:exactT1} reads as

\[0\rightarrow I_S/I^2_S\rightarrow\Omega^1_{M(S)}\otimes_{\raisebox{0.5pt}{\scriptsize $M(S)$}}R\rightarrow\Omega^1_{R}\rightarrow 0.\]
Therefore, one can see that (see Tjurina \cite{Tjurina})
\[
T_R^1=\frac{\mathbb{C}[x_1,\dots,x_g]^{g-1}}{\left(\frac{\partial f_i}{\partial x_j}\right)_{i,j}\mathbb{C}[x_1,\dots,x_g]^{g}+(f_1,\dots,f_{g-1})\mathbb{C}[x_1,\dots,x_g]^{g-1}}.
\]

From this expression of \(T_R^1,\) one can naturally try to use Delorme's resutls \cite{delormeglue} (see also Section \ref{sec:CI}) to describe \(T_R^1\) in a more detailed way. Surprisingly, until our paper together with Moyano-Fernández \cite{AMminiversal} there has been no attempt to apply Delorme's decomposition of the semigroup algebra of a monomial curve to describe \(T_R^1.\) Following our paper \cite{AMminiversal}, from the decomposition of the semigroup algebra explained in section \ref{sec:CI}, it is easily deduced that the Jacobian matrix presents a simple-to-describe block decomposition. Indeed, as in section \ref{subsec:poincaresemigroup}, let us assume that \(S\) is the gluing of \(S_1\) and \(S_2\) and let us assume that \(I_{S_1}=(h^1_1,\dots,h^1_{g_1-1})\) and \(I_{S_2}=(h^2_1,\dots,h^2_{g_2-1})\) are the defining ideals of \(S_1\) and \(S_2\) respectively, i.e.
\begin{equation*}
	\begin{split}
		\mathbb{C}[x_1,\dots,x_{g_1}]^{g_1-1}\rightarrow\mathbb{C}[x_1,\dots,x_{g_1}]\xrightarrow{\varphi_1} \frac{\mathbb{C}[x_1,\dots,x_{g_1}]}{I_{S_1}}\simeq \mathbb{C}[S_1]\rightarrow0\\
		\mathbb{C}[y_{1},\dots,y_{g_2}]^{g_2-1}\rightarrow\mathbb{C}[y_{1},\dots,y_{g_2}]\xrightarrow{\varphi_2} \frac{\mathbb{C}[y_1,\dots,y_{g_2}]}{I_{S_2}}\simeq \mathbb{C}[S_2]\rightarrow 0
	\end{split}
\end{equation*}
As we are assuming \(S\) being the gluing of the semigroups \(S_1\) and \(S_2,\) we have a binomial \(\rho\) in separated variables such that if \(\lambda_\rho\) denotes multiplication by \(\rho\) in \(\mathbb{C}[S_1]\otimes\mathbb{C}[S_2]\) and \mbox{ \(\pi:\mathbb{C}[S_1]\otimes\mathbb{C}[S_2]\rightarrow\operatorname{coker}(\lambda_\rho)\)} denotes the canonical projection, then the following diagram commutes

\begin{equation*}
	\begin{tikzcd}[
		column sep=scriptsize,row sep=scriptsize,
		ar symbol/.style = {draw=none,"\textstyle#1" description,sloped},
		isomorphic/.style = {ar symbol={\cong}},
		]
		\mathbb{C}[u_1,\dots,u_g]\ar[r]& \mathbb{C}[C^{S}]\ar[r]&0\\
		\mathbb{C}[x_1,\dots,x_{g_1}]\otimes \mathbb{C}[y_1,\dots,y_{g_2}]\ar[u,isomorphic]\ar[r,"\varphi_1\otimes\varphi_2"]&  \mathbb{C}[S_1]\otimes\mathbb{C}[S_2]\ar[u,"\pi\circ \lambda_\rho"]\ar[r]&0.
	\end{tikzcd}
\end{equation*}
Therefore, the Jacobian matrix for \(C^S\) has the form
\[\left(\frac{\partial f_i}{\partial u_j}\right)_{\raisebox{0.5pt}{\scriptsize $ \begin{array}{c}
		1\leq i\leq g-1 \\
		1\leq j\leq g 
\end{array}$}}=\left(\begin{array}{cc}
	\left(\frac{\partial h^1_i}{\partial x_j}\right) & 0 \\
	0 & \left(\frac{\partial h^2_i}{\partial y_j}\right)
	\\
	\rho_1& \rho_2
\end{array}\right ),\]
where we identify \(f_1=h_1^1,f_2=h_2^1,\dots,f_{g_1-1}=h_{g_1-1}^1,\) \(f_{g_1}=h_1^2,f_{g_1+1}=h_2^2,\dots,f_{g_1+g_2-1}=h_{g_2-1}^2\) and \(f_{g-1}=\rho_1+\rho_2\) for \(\rho_1=(\partial\rho/\partial u_1,\dots,\partial\rho/\partial u_{g_1})\) and \(\rho_2=(\partial\rho/\partial u_{g_1+1},\dots,\partial\rho/\partial u_{g}).\)
The consideration of this block decomposition leads us \cite{AMminiversal} to show the following theorem:

\begin{theorem}\label{thm:injective}\cite{AMminiversal} Let \(S\) be a complete intersection numerical semigroup. Let us define the submodules
	
	\[\overline{N}_S=\left\langle\Big(\frac{\partial f_1}{\partial u_1},\dots,\frac{\partial f_{g-1}}{\partial u_1}\Big),\dots,\Big(\frac{\partial f_1}{\partial u_g},\dots,\frac{\partial f_{g-1}}{\partial u_g}\Big)\right\rangle\subset \mathbb{C}[u_1,\dots,u_g]^{g-1}.\]
	\[
	\begin{split}
		\overline{N}_{S_1}:=\left\langle\left(\frac{\partial h^1_1}{\partial x_1},\dots,\frac{\partial h^1_{g_1-1}}{\partial x_1}\right),\dots,\left(\frac{\partial h^1_1}{\partial x_{g_1}},\dots,\frac{\partial h^1_{g_1-1}}{\partial x_{g_1}}\right)\right\rangle\subset\mathbb{C}[x_1,\dots,x_{g_1}]^{g_1-1},&\\
		\overline{N}_{S_2}:=\left\langle\left(\frac{\partial h^2_1}{\partial y_1},\dots,\frac{\partial h^2_{g_2-1}}{\partial y_1}\right),\dots,\left(\frac{\partial h^2_1}{\partial y_{g_2}},\dots,\frac{\partial h^2_{g_2-1}}{\partial y_{g_2-1}}\right)\right\rangle\subset\mathbb{C}[y_1,\dots,y_{g_2}]^{g_2-1}.&
	\end{split}
	\]
	For each \(i=1,2\) denote the canonical projections as
	\[\tau_1:\mathbb{C}[S]^{g-1}=\bigoplus_{j=1}^{g-1}\mathbb{C}[S]\  \mathbf{e}_j\rightarrow \mathbb{C}[S]^{g_1-1}=\bigoplus_{j=1}^{g_1-1}\mathbb{C}[S]e_j\quad \tau_2:\mathbb{C}[S]^{g-1}\rightarrow \mathbb{C}[S]^{g_2-1}=\bigoplus_{j=g_1+1}^{g_1+g_2-1}\mathbb{C}[S] \ \mathbf{e}_j,\]
	where the $\mathbf{e}_j$ build the standard $\mathbb{Z}$-basis. Write \(N_S:=\varphi(\overline{N}_S)\) and \(N_i:=\varphi_i(\overline{N}_{S_i})\) for $i=1,2$. Then, the linear maps \(\tau_1,\tau_2\) induce the following injective morphisms 
	$$
	\Phi_1 : \mathbb{C}[S_1]^{g_1-1}/ N_{S_1} \longrightarrow \mathbb{C}[S]^{g-1}/ N_{S},\quad
	\Phi_2 : \mathbb{C}[S_2]^{g_2-1}/ N_{S_2} \longrightarrow \mathbb{C}[S]^{g-1}/ N_{S},
	$$
\end{theorem}

Observe that Theorem \ref{thm:injective} implies that the base spaces of the miniversal deformations of \(C^{S_1}\) and \(C^{S_2}\) are embedded in the miniversal deformation of \(C^{S}.\) This is quite a specific and intrinsic property of a complete intersection monomial curve. In general, if we have a non-monomial complete intersection curve singularity, the defining ideal does not necessarily split in this way, thus this kind of decomposition in separated variables is not available.
\medskip

In \cite{AMminiversal}, we use Theorem \ref{thm:injective} to explicitly compute a monomial basis of \(T_R^1\) in the case of a free semigroup, generalizing a previous result of Cassou-Nogués \cite{pierette} who computed this monomial basis in the case of the semigroup of an irreducible plane curve singularity. Moreover, thanks to this explicit monomial basis, we are able \cite{AMminiversal} to provide lower and upper bounds for the dimension of the Pinkham's moduli space in terms of the generators of the semigroup. On the other hand, in 1993 St\"{o}hr \cite{stohr1} provided a rather explicit description of the moduli space of irreducible projective pointed Gorenstein curves with fixed Weiertrass semigroup which contains Pinkham's moduli space. His method is based in a careful analysis of the syzygies of the defining ideal of the curve together with an explicit computation of their Gr\"{o}bner basis. It would be desirable to explore whether Delorme's decomposition of the semigroup algebra of a complete intersection may help to better understand the structure of this moduli space.

\subsection{Deformations with constant Milnor number}
Given a germ \((C,x_0)\subset(\mathbb{C}^{n},x_0)\) of curve singularity, if \(D\subset \mathbb{C}\) is a small open disc with center \(0\) then a flat family is a deformation of \((C,x_0)\) of the form \(f:(X,x_0)\rightarrow (D,0)\) such that \((C,x_0)=(f^{-1}(0),x_0)=(X_0,x_0)\) with a section \(\sigma:D\rightarrow X.\) If we denote by \((X_t,\sigma(t))\) the fibers of this family, one wants to compare these germs of curves from a topological point of view. In order to do that, first we need to recall the definition of a "good representative" of a flat family as in \cite[Sec. 2.1]{BG80}. The curve \(C\) is assumed to be embedded in a small open ball \(B_0\subset \mathbb{C}^n\) with center \(x_0.\) \(X\) is a closed analytic subset of \(B=B_0\times D\) and \(f:X\rightarrow D\) is the restriction of the projection on \(D.\) Let \(C(f)\subset X\) denote the set of critical points of \(f\) and identify \(B_0\) with \(B_0\times\{0\}.\) By a good representative we are assuming that \(B_0\) is sufficiently small and \(D\) is sufficiently small with respect to \(B_0\) in a such a way the following conditions holds:
\begin{itemize}
	\item \(X\) and \(X_0\) are contractible and \(X_0\setminus\{x_0\}\) is nonsingular,
	\item \(f:X\rightarrow D\) is flat and \(f_{\mid C(f)}:C(f)\rightarrow D\) is finite,
	\item \(\partial B_0\times \{t\}\) intersects \(X_t\) transversally in regular points of \(X_t\) for all \(t\in D\) and each sphere \(S^{2n-1}_\epsilon\subset B_0\) with center \(x_0\) intersects \(X_0\) transversally.
\end{itemize}

We already mentioned that the definition by Buchweitz and Greuel \cite{BG80} of the Milnor number for reduced curves satisfies similar properties to those of the Milnor number of reduced plane curves. The most important is referred to the following characterization of topological triviality in a flat family.
\begin{theorem} \cite[Theorem 5.2.2]{BG80}
	Let \(f:X\to D\) be a good representative of a flat family of reduced curves with section \(\sigma:D\to X\) such that \(X_t\setminus\sigma(t)\) is smooth for each \(t\in D.\) The following conditions are equivalent:
	\begin{enumerate}
		\item \(\mu(X_t,\sigma(t))\) is constant for \(t\in D.\)
		\item \(\delta\) and the number of branches is constant.
		\item There exists a homeomorphism between \((B_0,X_0)\) and \((B_0,X_t)\) for each \(t\in D.\)
		\item There is a homeomorphism \(h:X\rightarrow X_0\times D\) such that \(f=\pi\circ h\) where\\ \(\pi:X_0\times D\rightarrow D\) is the projection, i.e. \(f:X\rightarrow D\) is topologically trivial.
	\end{enumerate}
\end{theorem} 
It is now natural to ask which are the possible value semigroups that can appear as value semigroups of the fibers in a flat family. As "generically" the fibers of the family are complete intersections and the Milnor number equals the conductor of the semigroup, then this problem is equivalent to that of finding the set of all possible complete intersection numerical semigroups with fixed conductor.
\medskip

In \cite{AGci13}, Assi and García-Sánchez provided an algorithm to construct all possible complete intersection numerical semigroups with fixed conductor with the help of Delorme's characterization. They already implemented this algorithm in GAP \cite{GAP4}. Unfortunately, the method provided in \cite{AGci13} does not allow us to provide a good description of how the evolution of the different semigroups appear as a deformation of a prescribed semigroup. It would be certainly an interesting problem to provide a more geometric classification of complete intersection numerical semigroups with fixed Milnor number. This is in fact closely related to the questions posed by Greuel and Buchweitz \cite{BG80}.

\medskip


\printbibliography

\end{document}
